\documentclass[aos]{imsart}

\RequirePackage{amsthm,amsmath,amsfonts,amssymb,mathrsfs}
\RequirePackage[numbers,sort&compress]{natbib}
\RequirePackage[colorlinks,citecolor=blue,urlcolor=blue]{hyperref}
\RequirePackage{graphicx}
\RequirePackage{tikz,booktabs}
\startlocaldefs
\theoremstyle{plain}

\newtheorem{theorem}{Theorem}[section]
\newtheorem{lemma}[theorem]{Lemma}
\newtheorem{corollary}[theorem]{Corollary}

\theoremstyle{definition}

\endlocaldefs

\begin{document}

\begin{frontmatter}
\title{The joint density of top k eigenvectors and principal subspace inference}
\runtitle{Top k principal eigenvectors}

\begin{aug}
\author[A]{\fnms{Koki}~\snm{Shimizu}\ead[label=e1]{k-shimizu@rs.tus.ac.jp}
\orcid{0000-0001-5201-4218}}
\author[B]{\fnms{Haoming}~\snm{Wang}\thanks{\textbf{Corresponding author}}\ead[label=e2]{wanghm37@nankai.edu.cn}\orcid{0009-0008-5023-6309}}
\address[A]{Department of Applied Mathematics,
Tokyo University of Science \printead[presep={,\ }]{e1}}

\address[B]{Center for Combinatorics,
Nankai University \printead[presep={,\ }]{e2}}
\end{aug}

\begin{abstract}
In this paper, the exact joint density of the top $k$ sample eigenvectors is derived for any $p\times p$ population covariance matrix $\varSigma$, sample size $n > p -1$, and $1 \le k \le p$. Using symmetric functions such as zonal polynomials and a dual summation identity by I. G. Macdonald, the result is expressed as a series in terms of determinants of differential operators. A matrix Kummer transformation converts the resulting local alternating expansion into a globally absolutely convergent series over the entire positive definite matrix space. For \(k=2\), the ordered eigenvalue integrals reduce to the Gauss hypergeometric function \({}_2F_1\) with a simple closed form when $n=p+1$. Previously, explicit formulas were only available for a scalar matrix $\varSigma = \sigma^2 I_p$ or a general matrix $\varSigma$ with $p=2$ or $k=1$, obtained by T. W. Anderson and T. Sugiyama, respectively. The frame law also induces an exact Grassmann density that provides a finite-sample benchmark for principal subspace inference.
\end{abstract}



\end{frontmatter}

\section{Introduction}\label{sec:introduction}

Principal Component Analysis (PCA) is a basic tool in classical multivariate statistical analysis and remains one of the most widely used methods in dimension reduction. The distribution theory of sample eigenvalues from a multivariate normal population is well developed, while the exact laws for the corresponding eigenvectors are available only in a few special cases. In particular, T. W. Anderson \cite{anderson1951asymptotic} and T. Sugiyama \cite{Sugiyama1965} studied the distribution of sample eigenvectors for a scalar population covariance matrix ($\varSigma = \sigma^2 I$) and a bivariate normal population, respectively. This gap matters whenever inference concerns several principal directions simultaneously, as in the construction of simultaneous confidence regions or confidence cones for covariance eigenspaces \cite{Tyler1981, Beran1988, NaumovSpokoinyUlyanov2019, JirakWahl2024}. The sample directions are constrained to be mutually orthogonal, and if the population covariance matrix is non-isotropic
(\(\varSigma\neq\sigma^2I\)), they are neither independent nor
rotationally invariant. This raises the natural question of whether one can derive the joint density of the leading sample eigenvectors for an arbitrary population covariance matrix.


\subsection{Statistical setting and previous work}

Suppose our data matrix $X = [{\rm x}_1;{\rm x}_2;\dots;{\rm x}_n]$ consists of $n$ independent identically distributed $p$-dimensional row vectors ${\rm x}_i$, $i=1,2,\dots,n$, each distributed according to a multivariate normal distribution $ N_p(0,\varSigma)$, where $\varSigma$ is a $p\times p$ real symmetric positive definite matrix. The fundamental goal of PCA is to find the top $k$ eigenvectors according to the $1\dots k$ largest eigenvalues of the sample covariance matrix $S = X'X$, namely the Wishart matrix, 
\[
h_j=\underset{\substack{\|h\|=1\\h\perp h_1,\ldots,h_{j-1}}}
{\operatorname{argmax}}h'Sh,\quad j=1,\ldots,p,
\]
and to determine the number \(k\) of principal components to be retained. Let $\lambda_j = h_j'Sh_j$. A common explained-variance rule selects the smallest $k$ such that
\[
\frac{\sum_{j=1}^k\lambda_j}{\sum_{j=1}^p\lambda_j}\ge 1-\alpha,
\quad 0<\alpha<1.
\]
Assuming $n> p-1$, the sample covariance matrix $X'X$ lies in a subspace with co-dimension $\ge 1$ with probability zero due to continuous density, so we can assume there are no repeated eigenvalues and $\lambda_1>\lambda_2>\dots \lambda_p > 0$. Each $h_j$ is determined only up to sign. Throughout, the density is understood under the angular and sign convention introduced in Section \ref{sec:orthogonal}. If only the principal subspace is of interest, the sign-invariant object is the projection $\sum_{j=1}^k h_jh_j'$.

\begin{figure}[!ht]
    \centering
    \includegraphics[width=1\linewidth]{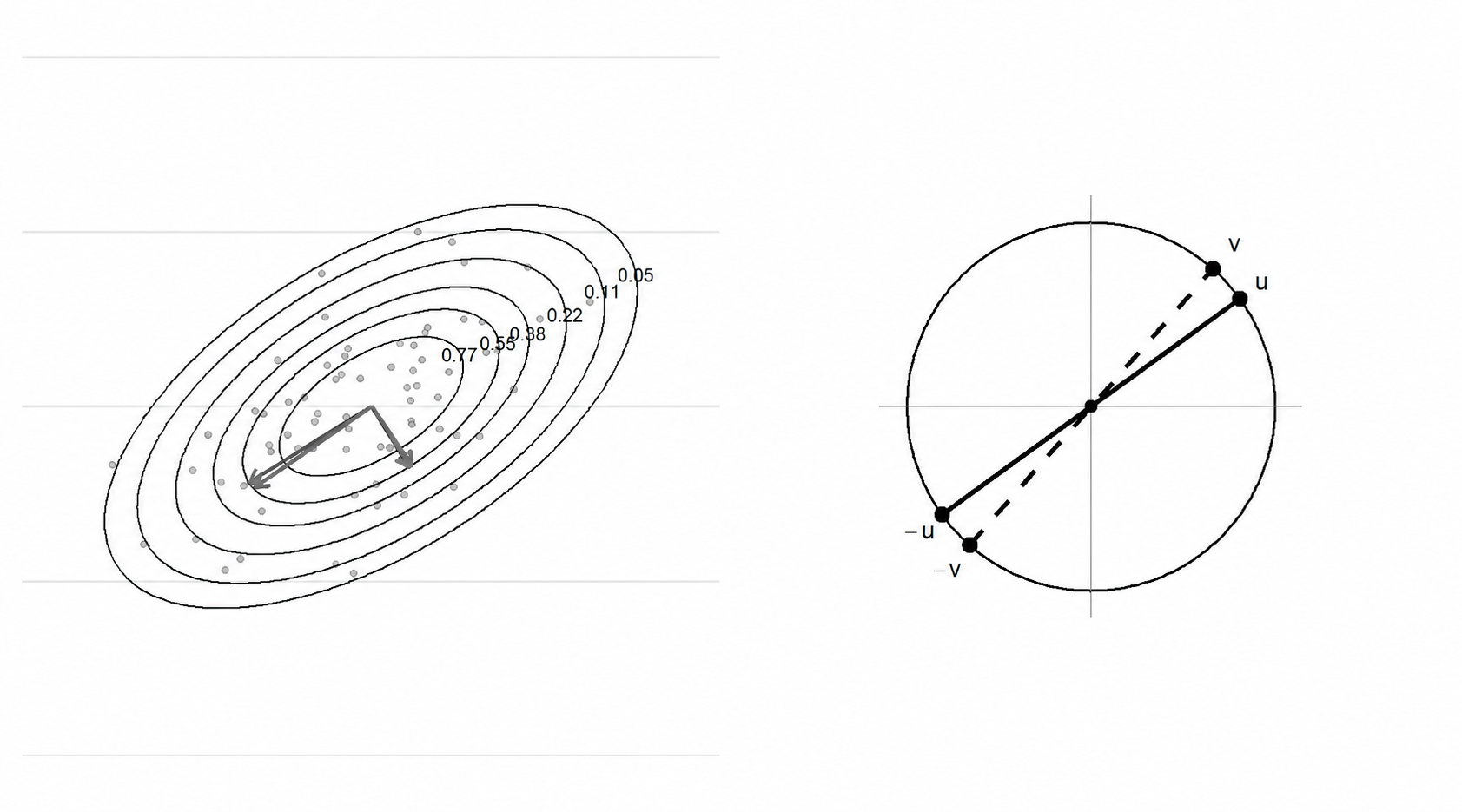}
    \caption{Left: contours of a bivariate Gaussian density and principal directions estimated from 70 observations. Black arrows indicate the population principal directions and standard deviations, and grey arrows indicate their sample estimates. Right: the Grassmannian $\mathbb{G}(2,1)$ identifies each line through the origin with a pair of antipodal points on $\mathbb{S}^1$.}
    \label{fig:pca2d}
\end{figure}

Thus, the question of PCA is two-fold. The joint distribution of eigenvalues of a sample covariance matrix is then guaranteed by a folklore lemma in random matrix theory, for example, \cite[Lemma 3.2.17]{muirhead1982aspects}. This lemma states that if $S$ is an $p \times p$ positive definite random matrix with continuous probability density $f(S)$, then the joint density function of the eigenvalues $\lambda_1, \dots, \lambda_p$ of $S$, arranged in decreasing order $\lambda_1 > \lambda_2 > \dots > \lambda_p> 0$ is\[\begin{aligned}    \frac{\pi^{{p^2}/{2}}}{\varGamma_p(p/2)}\prod_{1\le i < j \le p}(\lambda_i - \lambda_j) \int_{O(p)} f(H\varLambda H') [dH], \end{aligned}\] where $\varLambda = {\rm diag}(\lambda_i)$ is the diagonal matrix in eigenvalues $\lambda_1, \dots, \lambda_p$ of $S$, $\varGamma_p(a) = \pi^{{p(p-1)}/{4}}\prod_{i=1}^{p} \varGamma \left(a - {(i-1)}/{2}\right)$ is the $p$-variate Gamma function defined for $a > (p-1)/2$, and the integration is taken over the compact group of $p\times p$ orthogonal matrices $O(p) = \{H’H = HH' =  I_p\}$ with respect to the normalized Haar measure $[dH]$. 

From this, we can see that the change of variables from a symmetric matrix to its ordered eigenvalues and eigenvectors introduces the Vandermonde factor
\[
\varDelta(\varLambda)=\prod_{1\le i<j\le p}(\lambda_i-\lambda_j).
\]
The Vandermonde factor makes many eigenvalue marginals accessible through invariant integration and symmetric polynomial expansions, and it underlies a large literature on exact and asymptotic laws for Wishart roots \cite{james1960distribution,constantine1963some,khatri1972exact,johnstone2001distribution}.

The situation for eigenvectors is more delicate. If $\varSigma=\sigma^2I_p$, early investigations by T. W. Anderson \cite{anderson1951asymptotic} established that eigenvectors are uniformly distributed over the orthogonal group $O(p)$, subject to the chosen sign convention, and are independent of the ordered eigenvalues. For a general $\varSigma$, this rotational symmetry is broken, and the sample eigenvectors concentrate around the population principal directions. Classical asymptotic theory was developed by T. W. Anderson \cite{anderson1951asymptotic,anderson1963asymptotic} and further refined by Sugiura \cite{SUGIURA1976500}, Bai, Miao, and Pan \cite{Bai2007}.  Exact finite-sample results are much rarer. Sugiyama
obtained formulas for the cases \(p=2\) \cite{Sugiyama1965} and
\(k=1\) \cite{Sugiyama1966}. The exact joint density of an arbitrary leading $k$-frame, however, has not previously been available.

\subsection{The coupling problem and our contribution}
For $k\ge2$, the difficulty is not merely that the leading sample eigenvectors must be treated simultaneously, leading to $p(p-1)/2$-dimensional orthogonality constraint. To obtain the marginal law of the top $1< k < p$ eigenvectors, one must integrate out both the remaining eigenvectors and the nuisance eigenvalues
\[
0<\lambda_p<\cdots<\lambda_{k+1}<\lambda_k.
\]
The Vandermonde factor naturally decomposes as
\[
\varDelta(\varLambda)
=
\underbrace{\prod_{1\le i<j\le k}(\lambda_i-\lambda_j)}_{\varDelta_{11}}
\underbrace{\prod_{k\le i<j\le p}(\lambda_i-\lambda_j)}_{\varDelta_{22}}
\underbrace{\prod_{\substack{1\le i< k<j\le p}}(\lambda_i-\lambda_j)}_{\varDelta_{12}}.
\]
The mixed term $\varDelta_{12}$ couples the leading and nuisance eigenvalues and prevents the nuisance integral from factorizing. In root-system terminology, this is an $A_{p-1}$ interaction. Standard Selberg-Jack integrals \cite{KADELL199733,warnaar2009} handle the within-block terms, but not the mixed factor in this form. These
integrals are addressed by a dual sum identity by I. G. Macdonald \cite{Macdonal1992} and a determinant of commuting fractional differential operators $\mathscr D_i^{z}$, characterized by
\(\mathscr D_i^ze^{-q_i\lambda_i/2}=\lambda_i^z
e^{-q_i\lambda_i/2}\). These operators act on the ordered exponential
kernel
\[
\varPhi_k(\underline q)
=
\frac{2^k}{
q_1(q_1+q_2)\cdots(q_1+\cdots+q_k)
}, \quad q_{0j}=h_j'\varSigma^{-1} h_j, j =1,2, \dots, k.
\]

Our derivation resolves this coupling in three stages. First, we integrate over the remaining orthogonal degrees of freedom and express the result as a series of zonal polynomials. Second, after separating the factors involving $\lambda_1, \dots, \lambda_k$, we apply Macdonald's dual identity for Jack polynomials to expand the remaining mixed Vandermonde term. The nuisance eigenvalues can then be integrated by a Selberg-Jack formula. Third, we transform the integral over the ordered top eigenvalues to spacing coordinates and represent its powers by commuting fractional differential operators. A Schur expansion packages these operators into a determinant.

\subsection{Comparison with Sugiyama's formula}

In \cite{Sugiyama1966}, the result of Sugiyama states for some parametrization of the primary eigenvector $$h_1 = \Bigl(\cos \theta_{12}, \sin \theta_{12}\cos \theta_{13}, \dots, \prod_{j=2}^{p-1} \sin \theta_{1j} \cos\theta_{1p}\Bigr)'$$ and an orthogonal matrix $H_1$ with $h_1$ the first column, the density is given by
\begin{equation}
    \begin{aligned}	& |\varSigma|^{-{n}/{2}}\frac{\varGamma\left({n}/{2}\right)}{\varGamma_p\left({n}/{2}\right)}\cdot B_{p-1}\left(\frac{p+2}{2}, \frac{n-1}{2}\right)	\sum_{m=0}^{\infty} \frac{(pn/2)_m}{m!} 	\sum_{\kappa \vdash m} \frac{\left[\frac{1}{2}(n-1)\right]_\kappa}{\left[\frac{1}{2}(n+p+1)\right]_\kappa} \\ & \qquad \qquad \times \ C_{\kappa}\!\left(-\frac{\varSigma_{p-1}}{h_1' \Sigma^{-1} h_1}\right) (h_1' \varSigma^{-1} h_1)^{-pn/2}	\prod_{j=2}^{p-1} \sin^{p-j}\theta_{1j} \bigwedge_{j=2}^{p} d\theta_{1j},\end{aligned}\label{eq:sugiyama66}
\end{equation}
where $B_p(a,b) = {\varGamma_p(a)\varGamma_p(b)}/{\varGamma_p(a+b)}$ is the $p$-variate Beta function, $(a)_{\kappa,\alpha} = \prod_{i=1}^{p} (a - {(i-1)}/{\alpha})_{\kappa_i}, (a)_{\kappa} = (a)_{\kappa,2},$ $\varSigma_{p-1} = \left(H_1\varSigma^{-1}H_1\right)_{2:p;2:p}$ is obtained by 
\[
H_1\varSigma^{-1}H_1= \begin{pmatrix}
    * & *\\
    * & \varSigma_{p-1}
\end{pmatrix},\]
and $C_{\kappa}$ the zonal polynomial indexed by the integer partition $\kappa$. 
However, the above density function is expressed as an alternating series and may not converge numerically over the domain of definition. \cite{Ishizaki2009} derived a positive-term series representation through the Kummer transformation and also provided graphs of the density function for $p = 3$. 

In this paper, we will show that for some parametrization $\theta_{ij},1\le i < j \le p,$ of the orthogonal group $O(p)$ and $h_1,h_2,\dots,h_k$ are the first $k$ columns in the $p\times p$ matrix $H_k$ where $$H(\theta_{ij}) = H_k\begin{pmatrix} I_k & 0\\0 & H_{p-k}\end{pmatrix},$$ the joint probability density function of $h_1,h_2,\dots,h_k$ is given by a globally absolutely convergent series for every positive definite covariance matrix
\begin{align*}
&\frac{B_m(a,b)}
{2^{np/2}\varGamma_p(n/2)|\varSigma|^{n/2}}
\sum_{r=0}^{\infty}
\sum_{\substack{\mu\vdash r\\\ell(\mu)\leq m}}
\frac{C_\mu(\tfrac12\varSigma_m)}{r!C_\mu(I_m)}
\sum_{s=0}^{m(k-1)}\frac{1}{s!}
\sum_{\substack{\nu\vdash s,\ \nu_1\leq k-1\\
\ell(\nu)\leq m}}\widetilde  B_{\mu\nu}
\sum_{\substack{\sigma\vdash s,\ \sigma\leq\nu'\\
\ell(\sigma)\leq k-1}}
\notag\\
&\qquad\times
K_{\nu'\sigma}(1/2) \prod_{i=1}^k\prod_{j=i+1}^{p-1}
\sin^{p-j}\theta_{ij}
\bigwedge_{i=1}^k\bigwedge_{j=i+1}^{p}d\theta_{ij}\times 
\mathscr D_k^{v+u+m+r+s}
\prod_{i=1}^{d}\mathscr D_i^v
\notag\\
&\qquad\times
\begin{vmatrix}
(\mathscr D_1-\mathscr D_k)^{m+d-\sigma_d}
& (\mathscr D_1-\mathscr D_k)^{m+d-1-\sigma_{d-1}}
& \cdots
& (\mathscr D_1-\mathscr D_k)^{m+1-\sigma_1}\\
(\mathscr D_2-\mathscr D_k)^{m+d-\sigma_d}
& (\mathscr D_2-\mathscr D_k)^{m+d-1-\sigma_{d-1}}
& \cdots
& (\mathscr D_2-\mathscr D_k)^{m+1-\sigma_1}\\
\vdots & \vdots & \ddots & \vdots\\
(\mathscr D_d-\mathscr D_k)^{m+d-\sigma_d}
& (\mathscr D_d-\mathscr D_k)^{m+d-1-\sigma_{d-1}}
& \cdots
& (\mathscr D_d-\mathscr D_k)^{m+1-\sigma_1}
\end{vmatrix}
\varPhi_k(\underline{\tilde q})
\Big|_{\underline{\tilde q}=\underline{\tilde q}_0},
\end{align*}
where $d=k-1, m=p-k, a = (n - k)/2, b = (p-k + 3)/2, u=(p-k)(n-k)/2 ,  v = {(n - p-1)}/{2}$, \(\tilde q_{0k}=q_{0k}+\operatorname{tr}\varSigma_m\), $\varSigma_{m} = \left(H_k'\varSigma^{-1}H_k\right)_{k+1:p;k+1:p}$ is obtained by
\[
H_k'\varSigma^{-1}H_k= \begin{pmatrix}
    * & *\\
    * & \varSigma_{m}
\end{pmatrix}.\]
In particular, \(\nu'\) is the conjugate partition of \(\nu\), the coefficients \(K_{\nu'\sigma}(1/2)\) are defined by the Jack-Schur expansion, and $\widetilde B_{\mu\nu}$ is given by the Littlewood-Richardson coefficients $g_{\mu\nu}^{\phi}$ for zonal polynomials,
\[
\begin{aligned}
\widetilde B_{\mu\nu}
=
\sum_{\substack{
\phi\vdash|\mu|+|\nu|\\
\ell(\phi)\leq m
}}
& g_{\mu\nu}^{\phi}
\frac{(b)_\phi}{(a+b)_\phi}
C_\phi(I_m).
\end{aligned}\]

The resulting density is valid for $1\le k\le p$ and arbitrary positive definite $\varSigma$. For \(k=1\), the dual sum disappears and 
it reduces to a globally convergent Kummer transform of Sugiyama's formula when $k=1$ and provides a finite-sample basis for simultaneous inference on the leading principal subspace.  

This paper is organized as follows. Sections~\ref{sec:symmetric-functions} and~\ref{sec:jack-zonal} review the required symmetric function and Jack polynomial identities. Section~\ref{sec:fractional-operators} defines	the fractional operators, Section~\ref{sec:orthogonal} introduces the orthogonal	parametrization, and Section~\ref{sec:joint-density} derives both the local and globally convergent series for the joint density. Section~\ref{sec:k2-explicit} gives for $k=2$ the \({}_2F_1\) representation and the elementary specialization when \(n=p+1\). Section~\ref{sec:principal-subspace} then pushes the frame
	law to the Grassmannian and develops exact finite-sample tests,
	confidence regions, bootstrap diagnostics, and power comparisons for
	inference on principal subspaces.



\section{Partitions, Tableaux, and Symmetric Functions}
\label{sec:symmetric-functions}

Our distributional results rely on expansions of symmetric functions
indexed by integer partitions. This section fixes the notation and
reviews only the combinatorial facts used later.

\subsection{Partitions and Young diagrams}

An integer partition of \(n\), denoted by \(\lambda\vdash n\), is a
weakly decreasing sequence
\(\lambda=(\lambda_1,\ldots,\lambda_k)\) of nonnegative integers such
that \(\sum_{i=1}^k\lambda_i=n\). Its weight is
\(|\lambda|=n\), and its length \(\ell(\lambda)\leq k\) is the number
of positive parts.

The Young diagram of \(\lambda\) is the set of cells \((i,j)\) such
that \(1\leq i\leq\ell(\lambda)\) and
\(1\leq j\leq\lambda_i\). For example, the five partitions of \(4\)
are \((4)\), \((3,1)\), \((2,2)\), \((2,1,1)\), and
\((1,1,1,1)\).
\[
\begin{tikzpicture}[scale=0.6, baseline=(current bounding box.center)]
\draw (0,0) grid (4,1);
\node[below=0.2cm] at (2,0) {$\lambda=(4)$};
\end{tikzpicture}
\qquad
\begin{tikzpicture}[scale=0.6, baseline=(current bounding box.center)]
\draw (0,0) grid (3,1);
\draw (0,-1) rectangle (1,0);
\node[below=0.2cm] at (1.5,-1) {$\lambda=(3,1)$};
\end{tikzpicture}
\qquad
\begin{tikzpicture}[scale=0.6, baseline=(current bounding box.center)]
\draw (0,-1) grid (2,1);
\node[below=0.2cm] at (1,-1) {$\lambda=(2,2)$};
\end{tikzpicture}
\qquad
\begin{tikzpicture}[scale=0.6, baseline=(current bounding box.center)]
\draw (0,0) grid (2,1);
\draw (0,-2) grid (1,0);
\node[below=0.2cm] at (1,-2) {$\lambda=(2,1,1)$};
\end{tikzpicture}
\qquad
\begin{tikzpicture}[scale=0.6, baseline=(current bounding box.center)]
\draw (0,-3) grid (1,1);
\node[below=0.2cm] at (0.5,-3) {$\lambda=(1,1,1,1)$};
\end{tikzpicture}
\]
The conjugate partition $\lambda'$ is obtained by transposing the diagram, where $\lambda'_j$ is the number of cells in the $j$-th column. For example, the conjugate partition of $(3,1)$ is $(2,1,1)$.
\[\begin{tikzpicture}[scale=0.6, baseline=(current bounding box.center)]
\foreach \x in {0,1,2} \draw (\x,0) rectangle (\x+1,1);
\draw (0,-1) rectangle (1,0);
\node[right=0.5cm] at (3,0) {$\lambda=(3,1)$};
\end{tikzpicture}
\qquad
\begin{tikzpicture}[scale=0.6, baseline=(current bounding box.center)]
\foreach \x in {0,1} \draw (\x,0) rectangle (\x+1,1);
\draw (0,-1) rectangle (1,0);
\draw (0,-2) rectangle (1,-1);
\node[right=0.5cm] at (2,-0.5) {$\lambda'=(2,1,1)$};
\end{tikzpicture}\]
Algebraically, the $i$-th part of the conjugate partition, $\lambda'_i$, corresponds to the number of cells in the $i$-th column of the original diagram, or equivalently, the number of parts in $\lambda$ that are greater than or equal to $i$.

\subsection{Skew diagrams and tableaux}

For partitions \(\lambda\) and \(\mu\) of the same weight, write
\(\mu\leq_D\lambda\) if $\mu_1 + \mu_2 + \dots +\mu_i \le \lambda_1 + \lambda_2 + \dots + \lambda_i$ for all $i$. This is the dominance order. The lexicographic order is the total
order in which \(\mu<_{\mathrm{lex}}\lambda\) if, at the first index
where the parts differ, \(\mu_i<\lambda_i\). Dominance implies the
corresponding lexicographic order, but not conversely. For example,
\((3,3)<_{\mathrm{lex}}(4,1,1)\), while these two partitions of \(6\)
are incomparable in dominance order.

If the diagram of \(\mu\) is contained in that of \(\lambda\), the
skew diagram \(\lambda/\mu\) is obtained by removing the cells of
\(\mu\) from \(\lambda\).
\[\begin{tikzpicture}[scale=0.6, baseline=(current bounding box.center)]
\draw (1,0) rectangle (2,1);
\draw (2,0) rectangle (3,1);
\draw (1,-1) rectangle (2,0);
\draw (0,-2) rectangle (1,-1);
\node[below=0.5cm] at (1.5,-2) {$\nu=(3,2,1)/(1,1)$};
\end{tikzpicture}
\qquad
\begin{tikzpicture}[scale=0.6, baseline=(current bounding box.center)]
\draw (2,0) rectangle (3,1);
\draw (3,0) rectangle (4,1);
\draw (1,-1) rectangle (2,0);
\draw (2,-1) rectangle (3,0);
\node[below=0.5cm] at (1.5,-1) {$\nu=(4,3)/(2,1)$};
\end{tikzpicture}\]
A standard (semi-Standard) Young tableau of shape $\lambda / \mu$ is a mapping $T$ from the cells of the skew diagram to the set of positive integers $\mathbb{N}_{\ge1}$ that satisfies two conditions
\begin{enumerate}
    \item The entries are strictly (weakly) increasing along each row from left to right.
    \item The entries are strictly increasing down each column from top to bottom.
\end{enumerate}
Recording the number of times each number appears in a tableau gives a sequence known as the weight of the tableau. Thus, the Standard Young Tableaux (SYT) are precisely the Semi-Standard Young Tableaux (SSYT) of weight $(1,1,\dots,1)$, which requires every integer up to $n$ to occur exactly once.

\[\begin{tikzpicture}[scale=0.6]
    \begin{scope}[shift={(0,0)}]
        \draw (0,0) grid (3,-1);
        \draw (0,-1) grid (2,-2);
        \draw (0,-2) grid (1,-3);
        
        \node at (0.5,-0.5) {1}; \node at (1.5,-0.5) {1}; \node at (2.5,-0.5) {2};
        \node at (0.5,-1.5) {2}; \node at (1.5,-1.5) {3};
        \node at (0.5,-2.5) {3};
        \node[below=0.5cm] at (1,-2.5) {SSYT with weight $(2,2,2)$};
    \end{scope}
    \begin{scope}[shift={(8,0)}]
        \draw (0,0) grid (3,-1);
        \draw (0,-1) grid (2,-2);
        \draw (0,-2) grid (1,-3);
        
        \node at (0.5,-0.5) {1}; \node at (1.5,-0.5) {2}; \node at (2.5,-0.5) {5};
        \node at (0.5,-1.5) {3}; \node at (1.5,-1.5) {4};
        \node at (0.5,-2.5) {6};
        \node[below=0.5cm] at (1,-2.5) {SYT with weight $(1,1,\dots,1)$};
    \end{scope}
\end{tikzpicture}\]

\subsection{Symmetric groups over alphabet $1\dots n$}  The symmetric group \(\mathfrak S_n\) consists of the \(n!\) permutations of \(\{1,\ldots,n\}\). It enters the theory of symmetric functions through the Frobenius characteristic map and through its action on exponent sequences $x_1^{\alpha_1}\dots x_k^{\alpha_k}$. The stabilizer is defined as 
\(\operatorname{Aut}_N(\lambda)=\{\sigma\in\mathfrak S_N:
\sigma\lambda=\lambda\}\) for $\ell(\lambda)\le N$ by appending zeros to $\lambda$.

\subsection{Symmetric functions}

A symmetric function in variables \(x_1, x_2, \dots\) is a formal power series that is invariant under any permutation of the variables and has bounded degree. The ring of symmetric functions \(\Lambda\) is the graded ring consisting of such series.  The most common symmetric function bases are the following
\begin{itemize}
\item ({Monomial}) \(
m_\lambda = \frac{1}{|\mathrm{Aut}_N(\lambda)|} \sum_{\sigma \in \mathfrak{S}_N} x_{\sigma(1)}^{\lambda_1} \cdots x_{\sigma(N)}^{\lambda_N}\).

\item ({Power Sum}) $p_\lambda = \prod_{i=1}^{\ell(\lambda)} p_{\lambda_i}$ where $p_k = m_{(k)} =\sum_{i\ge 1} x_i^k.$

\item ({Elementary}) $e_\lambda = \prod_{i=1}^{\ell(\lambda)} e_{\lambda_i}$ where $ e_k = m_{(1^k)} = \sum_{1\le i_1<\cdots<i_k} x_{i_1}\cdots x_{i_k}$.

\item ({Complete Symmetric}) $h_\lambda = \prod_{i=1}^{\ell(\lambda)} h_{\lambda_i}$ where $ h_k = \sum_{|\kappa| = k} m_{\kappa} = \sum_{1\le i_1\le\cdots\le i_k} x_{i_1}\cdots x_{i_k}.$

\item ({Schur}) $s_\lambda = \det\bigl(h_{\lambda_i - i + j}\bigr)_{i,j} = \det\bigl(e_{\lambda'_i - i + j}\bigr)_{i,j}$.
\end{itemize}

The last two determinants are the Jacobi--Trudi identities. In
addition, for \(n\geq\ell(\lambda)\), the Weyl character formula gives \cite{macdonald1998symmetric}
\begin{equation}
s_\lambda(x_1,x_2,\dots,x_n)=
\frac{\det\bigl(x_i^{\lambda_j+n-j}\bigr)_{1\le i,j\le n}}
{\det\bigl(x_i^{n-j}\bigr)_{1\le i,j\le n}}. \label{lem:Weyl}
\end{equation}
Relations among these bases can be found in
\cite{goulden2004combinatorial}.

\subsection{Irreducible representations of $\mathfrak S_n$}
The irreducible complex representations of the symmetric group \(\mathfrak{S}_n\) are canonically indexed by integer partitions \(\lambda \ \vdash n\). For each \(\lambda\), the dimension \(f^\lambda\) of the corresponding irreducible representation \(\rho^\lambda\) is given by the hook-length formula
\[
f^\lambda = \frac{n!}{\prod_{(i,j) \in \lambda} h(i,j)},
\]
where \(h(i,j)\) is the number of cells to its right in the same row plus the number below it in the same column, counting the cell $(i,j)$ once within the Young diagram of \(\lambda\). 

These representations play a fundamental role in the theory of symmetric functions. By the Frobenius characteristic formula, the Schur function \(s_\lambda\) is the direct algebraic image of the irreducible character \(\chi^\lambda\)
\[
s_\lambda(x) = \frac{1}{n!} \sum_{\pi\in S_n} \chi^\lambda(\pi)   p_{\rho(\pi)}(x),
\]
where \(\rho(\pi)\) denotes the cycle type of the permutation \(\pi\), an integer partition of \(n\) whose parts are the lengths of the cycles in the disjoint cycle decomposition of \(\pi\), and $p_{\cdot}$ the power sum.

While the Frobenius formula provides the group-theoretic definition, the Schur polynomial concurrently admits a powerful combinatorial evaluation. Restricting to a finite number of variables \(x_1,\dots,x_m\), the Schur polynomial \(s_\lambda\) of degree \(|\lambda|\) is identically the generating function of restricted SSYT
\[
s_\lambda(x_1,\dots,x_m) = \sum_{T \in \mathrm{SSYT}(\lambda, m)} \prod_{i=1}^m x_i^{w_i(T)},
\]
where the sum is over all SSYT of shape \(\lambda\) with entries drawn from \(\{1,\dots,m\}\), and \(w_i(T)\) records the multiplicity of the integer \(i\) appearing in tableaux \(T\). 

\section{Jack Polynomials in the Real, Complex and Quaternion}\label{sec:jack-zonal}

	In the real case, zonal polynomials correspond under the
	characteristic map to the zonal spherical functions of the Gelfand
	pair \((\mathfrak S_{2n},\mathfrak H_n)\), where
	\(\mathfrak H_n=\mathfrak S_2\wr\mathfrak S_n\) is the
	hyperoctahedral group \cite{james1961zonal}. They also arise as
	spherical polynomials for the symmetric space
	\(GL_n(\mathbb R)/O(n)\). In the complex case, the analogous
	characteristic map is the Frobenius map
	\(\operatorname{ch}(\chi^\lambda)=s_\lambda\)
	\cite{macdonald1998symmetric}. Jack polynomials interpolate the
	spherical polynomial families associated with the real, complex, and
	quaternionic division algebras. In our convention, their Jack
	parameters are \(\alpha=2\), \(\alpha=1\), and \(\alpha=1/2\),
	respectively.

\subsection{Jack polynomial}
Let $\mathbb Q(\alpha)$ be the field of fractions of the polynomial ring $\mathbb Q[\alpha]$ with rational coefficients in $\mathbb Q$ and indeterminate $\alpha$. If a partition $\lambda$ has $m_i$ parts equal to $i$, then write
\[z_{\lambda} = (1^{m_1}2^{m_2}\dots)m_1!m_2!\dots.\]

Define a bilinear scalar product on the vector space of symmetric functions $\Lambda(\alpha)= \Lambda \otimes \mathbb Q(\alpha)$ with indeterminates in $x = (x_1,x_2,\dots)$ (possibly infinite) and coefficients in $\mathbb Q(\alpha)$ by this condition
\[\langle p_{\lambda}, p_{\mu}\rangle_\alpha = \delta_{\lambda\mu} z_{\lambda} \alpha^{\ell({\lambda})},\]
where $p_\lambda$ is the power-sum symmetric polynomial and $\ell(\lambda)$ is the length of $\lambda$. 

The system of J-normalised Jack symmetric polynomials $J_{\lambda} = J_{\lambda}(x;\alpha)$ satisfies the following axioms
\begin{enumerate}
    \item[(C1)] (Orthogonality) $\left\langle J_\lambda, J_\mu\right\rangle_\alpha=0$ if $\lambda \neq \mu$;
    \item[(C2)] (Triangularity) Expanding them in terms of monomials
    $$ J_\lambda(x;\alpha)=\sum_\mu v_{\lambda \mu}(\alpha) m_\mu(x),$$
    $v_{\lambda \mu}(\alpha)=0$ unless $\mu \leqslant \lambda$ in the dominance ordering.
    \item[(C3)] (Normalization) If $|\lambda|=n$, then the coefficient $v_{\lambda, 1^n}$ in $J_\lambda(x;\alpha)$ is $n!$. 
\end{enumerate}
Once (C1)-(C3) are fulfilled, the Jack polynomial is uniquely determined. See \cite{STANLEY198976} for instance. 

We need to mention the other two normalisations, namely P-normalisation and C-normalisation. The P-normalization is the standard monic normalization for combinatorial identities and duality, while the C-normalization is particularly suited to zonal polynomial expansions and hypergeometric functions of matrix argument in multivariate
statistics.

If $(i, j) \in \lambda$, then the quantity $h_\lambda(i, j)=h(i, j)= \lambda_i+\lambda_j^{\prime}-i-j+1$ is called the hook length at $(i, j)$. We define two $\alpha$-refinements of $h_\lambda(i, j)$ as follows
$$
\begin{aligned}
& h_\lambda^*(i, j)=h^*(i, j)=\lambda_J^{\prime}-i+\alpha\left(\lambda_i-j+1\right) \\
& h_*^{\lambda}(i, j)=h_*(i, j)=\lambda_j^{\prime}-i+1+\alpha\left(\lambda_i-j\right) .
\end{aligned}
$$
We call $h^*(i, j)$ the upper hook-length and $h_*(i, j)$ the lower hook length at $(i, j)$. 
The P-normalised Jack polynomial $P_{\lambda} = P_{\lambda}(x;\alpha)$ is related to $J_{\lambda} (x;\alpha)$ by 
$$J_{\lambda}(x;\alpha) = 
\left(\prod_{s \in \lambda}h^{\lambda}_*(s)\right) P_{\lambda}(x;\alpha)$$
Furthermore, the norm square of $J_\lambda$ is given by the formula
$$j_{\lambda}(\alpha) = \left\langle J_\lambda, J_\lambda\right\rangle_\alpha = \prod_{s \in \lambda}h^{\lambda}_*(s)h_{\lambda}^*(s),$$
so the \(C\)-normalized Jack polynomial $C_{\lambda} = C_{\lambda}(x;\alpha)$ is related to $J_{\lambda} (x;\alpha)$ by
\begin{equation*}
C_\lambda(x;\alpha)
=\frac{\alpha^{|\lambda|}|\lambda|!}{j_\lambda(\alpha)}
J_\lambda(x;\alpha).
\label{eq:Jack-normalizations}
\end{equation*}
When $\alpha = 1$, $s_{\lambda} (x) = P_{\lambda}(x;1)$ reduces to the Schur polynomial and when $\alpha =2$, $C_{\lambda}(x) = C_{\lambda}(x;2)$ is the classical zonal polynomial.

In the following, we need a lemma of I. G. Macdonald \cite{Macdonal1992}. Let $\lambda'$ be the conjugate of $\lambda$.

\begin{lemma}[Dual Identity] 
$\prod_{i,j}(1 + x_iy_j) = \sum_{s\ge0}({1}/{s!}) \sum_{\lambda \vdash s} C_{\lambda}(x;\alpha)J_{\lambda'}(y;1/\alpha).$\label{lem: dual cauchy}
\end{lemma}

We also use the product expansion for real zonal polynomials. 
\begin{lemma} \label{lem: LR rule} 
$C_{\lambda}(x)C_{\mu}(x) = \sum_{\nu  \, \vdash |\lambda|+|\mu|} g_{\lambda\mu}^{\nu}C_{\nu}(x)$.
\end{lemma}

The coefficients \(g_{\lambda\mu}^\nu\) are the zonal structure
constants or Littlewood-Richardson coefficients, 
{which appear in several multivariate testing problems in multivariate statistics, such as MANOVA and sphericity tests. For values up to the seventh degree of $|\lambda|$ and $|\mu|$, these were tabulated in \cite{khatri1968non}.}
The special cases where $\mu=(m)$ or $\mu = (1^m)$, the latter analogous to the Pieri rule, were established by \cite{kushner1988product} and \cite{vretare1985product}, respectively. More recently, an efficient algorithm for computing these coefficients in the general case was developed by \cite{shimizu2022algorithm}.

\subsection{Hypergeometric functions} The importance of introducing zonal polynomials is that they form an orthogonal invariant basis for hypergeometric functions and other transcendental functions in representing multivariate integrals. The multivariate Gamma function, denoted by $\varGamma_{p}(a)$, is defined to be
\begin{equation}
    \varGamma_{n}(a) = \int_{A>0} \operatorname{etr} (-A) |A|^{a-\frac{n+1}{2}} (dA), \qquad {\rm etr} = \exp {\rm tr}
\end{equation}
where $\Re (a) > \frac{1}{2}(n-1)$ and $A>0$ means the integral is taken over the space of real symmetric positive definite $n \times n$ matrices. The multivariate Beta function, denoted by $B_{n}(a,b)$, is defined to be
\begin{equation}
    B_{n}(a,b) =\int_{0< X < I_n} |X|^{a - \frac{n+1}{2}} |I-X|^{b - \frac{n+1}{2}} (dX),
\end{equation}
where $\Re (a), \Re (b) > \frac{1}{2}(n-1)$, and the integral is taken over all $n\times n$ real symmetric matrices $X$ such that both $X$ and $I - X$ are positive definite. The multivariate Beta function is related to the multivariate Gamma function by the formula 
\begin{equation}
    B_{n}(a,b) = \frac{\varGamma_{n}(a)\varGamma_{n}(b)}{\varGamma_{n}(a + b)}.
\end{equation}

For a partition $\kappa = (\kappa_1, \kappa_2, \dots, \kappa_n)$ of the integer $k$, the $\alpha$-refined Pochhammer symbol or rising factorial symbol  associated with the partition $\kappa$ is defined as
\begin{equation}
    (a)_{\kappa,\alpha} = \prod_{i=1}^{n} \left( a - \frac{i-1}{\alpha} \right)_{\kappa_i},\label{eq: Pochhammer symbol}
\end{equation}
where $(x)_k = x(x+1)\dots(x+k-1)$ and $(x)_0 = 1$. When $\alpha = 2$, $(a)_{\kappa} = (a)_{\kappa,2}$ reduces to the James' notation \cite{james1960distribution}. Here, we restrict attention to Jack polynomials $J_{\lambda}(x;\alpha)$ with $\alpha = 2,1,1/2$, hence corresponding to real, complex, and quaternionic cases.

The hypergeometric function on $\underline{x} = (x_1,x_2,\dots,x_n)$, denoted by ${}_pF_q(\underline x)$, is conjugate invariant on these normed division algebras $\mathbb R, \mathbb C, \mathbb H$ so they can be easily extended to a single real symmetric (or complex Hermitian etc) matrix argument with $\underline x$ as its eigenvalues. The function is defined by the series \cite{Vilenkin2010RepresentationOL}
\begin{equation}
    {}_pF_q(\underline{a}; \underline{b}; \underline{x};\alpha) = \sum_{k=0}^{\infty} \sum_{\kappa \ \vdash k} \frac{(a_1)_{\kappa,\alpha} \dots (a_p)_{\kappa,\alpha}}{(b_1)_{\kappa,\alpha} \dots (b_q)_{\kappa,\alpha}} \frac{C_{\kappa}(\underline{x};\alpha)}{k!},
\end{equation}
where $\underline{a} = (a_1, \dots, a_p)$, $\underline{b} = (b_1, \dots, b_q)$. The parameters $a_i$ and $b_j$ are complex constants such that no $b_j - ({i-1})/{\alpha}$ is a non-positive integer for $i=1, \dots, n$. 

Special cases for $\alpha = 2$ of this general form ${}_pF_q(\underline{a}; \underline{b}; \underline{x}) = {}_pF_q(\underline{a}; \underline{b}; \underline{x};2)$ are particularly significant in multivariate analysis when we are dealing with real symmetric matrices. There are two examples.
\begin{itemize}
    \item The matrix exponential ${}_0F_0(X)$ for an $n\times n$ real symmetric matrix $X$ is given by $${}_0F_0(X) = \operatorname{etr}(X) = \sum_{k=0}^{\infty} \sum_{\kappa \ \vdash k} \frac{C_{\kappa}(X)}{k!}.$$
    \item The matrix binomial expansion ${}_1F_0(X)$ for $\|X\| < 1$ is $${}_1F_0(a; X) = |I-X|^{-a} = \sum_{k=0}^{\infty} \sum_{\kappa \ \vdash k} (a)_{\kappa} \frac{C_{\kappa}(X)}{k!}.$$
\end{itemize}
Moreover, the series ${}_pF_q$ converges for all $X$ if $p \leq q$, and for $\|X\| < 1$ if $p = q+1$. These functions provide the analytical framework for deriving the distributions of the eigenvectors of the Wishart distribution. 

\subsection{A Selberg-Jack integral}

\begin{lemma}[{\cite[Theorem 1]{KADELL199733}, \cite[Theorem 1.2]{warnaar2009}}]
Let \(n\geq 1\) and let \(p_{\alpha}=1+(n-1)/\alpha\) where \(\alpha\in \{2,1,1/2\}\). Then for \(\Re(a)>(n-1)/{\alpha}\) and \(\Re(b)>(n-1)/{\alpha}\), the Selberg-Jack integral is
\[
\begin{aligned}
\int_0^1\dots \int_0^1
    \prod_{i=1}^{n}
    x_i^{a-p_\alpha}
    (1-x_i)^{b-p_\alpha}
        \prod_{i=1}^n
        \prod_{j=i+1}^n
        |x_i-x_j|^{{2}/{\alpha}}
    C_{\kappa}(x_1,\dots,x_n;\alpha)
    dx_1\dots dx_n\\
=
    S_n(a-p_\alpha+1,b-p_\alpha+1,1/\alpha)
    \frac{(a)_{\kappa,\alpha}}{(a+b)_{\kappa,\alpha}}
    C_{\kappa}(1,\dots,1;\alpha),
\end{aligned}
\]
where
\((a)_{\kappa,\alpha}\) is the Pochhammer symbol \eqref{eq: Pochhammer symbol} and the constant $S_n(A,B,C)$ is
\[
\prod_{j=0}^{n-1}
\frac{\varGamma(A+jC)\varGamma(B+jC)\varGamma(1+(j+1)C)}{\varGamma(A+B+(n+j-1)C)\varGamma(1+C)}.
\]
In particular, the right hand side is divided by \(n!\) when the integration is taken over \(1>x_1>x_2>\cdots>x_n>0.\)
\label{lem: Selberg}
\end{lemma}

This integral was systematically developed by several authors. 
For a historical discussion of the Selberg integral and its
extensions, see \cite{ForresterWarnaar2008}. We apply this identity in
Section~\ref{sec:joint-density}.

\section{Fractional Differential Operators and Laplace Transforms}
\label{sec:fractional-operators}

The differential operators used below arise from a simple property of
the Laplace transform. For every non-negative integer \(r\),
\[
2^r(-\partial_{q_i})^r
e^{-q_i\lambda_i/2}
=
\lambda_i^r e^{-q_i\lambda_i/2}.
\]
Since the powers of \(\lambda_i\) occurring in our eigenvalue integrals
need not be integers, extensions of this relation to arbitrary real or complex powers are needed.

Let
\(W_k
=
\bigl\{
\lambda\in(0,\infty)^k:
\lambda_1>\cdots>\lambda_k
\bigr\}\), and suppose that
\[
F_g(q)
=
\int_{W_k}
e^{-\frac12\sum_{j=1}^kq_j\lambda_j} g(\lambda)d\lambda,
\qquad q\in(0,\infty)^k,
\]
is absolutely convergent.

\begin{lemma}[{\cite[Section 2.2 and Section 2.9]{KilbasSrivastavaTrujillo2006}}]\label{lem:fractional-properties}
There exists a unique family of commuting operators \(\mathscr D_i^{\,z}, i=1,\ldots,k, z\in\mathbb R,\) on the class of Laplace transforms under consideration such that
\[
\mathscr D_i^{\,z}F_g(q)
=
\int_{W_k}
\lambda_i^z
e^{-\frac12\sum_{j=1}^kq_j\lambda_j}
g(\lambda)\,d\lambda,
\]
whenever the integral is absolutely convergent. These operators satisfy for $r\in\mathbb N_0$,
\[\mathscr D_i^{r}=2^r(-\partial_{q_i})^r, \qquad \mathscr D_i^{z}\mathscr D_i^{w}
=
\mathscr D_i^{z+w},
\quad
\mathscr D_i^{z}\mathscr D_j^{w}
=
\mathscr D_j^{w}\mathscr D_i^{z}, \quad 1\le i,j\le k.
\]
\end{lemma}

In the following, the entries of every operator matrix commute. If \(\mathfrak D=(\mathfrak D_{ij})_{i,j=1}^n\) is such a matrix, its operator determinant is defined by
\begin{equation}
\det(\mathfrak D)F
:= \begin{vmatrix}
\mathfrak D_{11} & \cdots & \mathfrak D_{1n}\\
\vdots           & \ddots & \vdots\\
\mathfrak D_{n1} & \cdots & \mathfrak D_{nn}
\end{vmatrix}F =
\sum_{\pi\in\mathfrak S_n}
\operatorname{sgn}(\pi)
\left(
\prod_{i=1}^n
\mathfrak D_{i,\pi(i)}
\right)F.
\label{eq:operator-determinant}
\end{equation}
The commutativity of the entries ensures that the order of the
operators in each product is irrelevant.

\section{Tumura's Trick for Orthogonal Group Decomposition}\label{sec:orthogonal}

In this section, we use $\underline{\theta}$ to denote vector and $\underline{\underline{\theta}}$ matrix. $dX$ is a symbol for exterior differential and $(dX)$ its wedge product defined as 
\[(dX) = \begin{cases}
  \displaystyle   \bigwedge_{i=1}^{n}\bigwedge_{j=1}^{p} dx_{ij}, & \text{ for $n\times p$ rectangular matrix } X;\\  \displaystyle   \bigwedge_{1 \le i\le j \le p} dx_{ij}, & \text{ for $p\times p$ real symmetric matrix } X;\\
  \displaystyle   \bigwedge_{1 \le i< j \le p} dx_{ij}, & \text{ for $p\times p$ real skew-symmetric matrix } X.
\end{cases}\]

For fixed positive integers $p$ and $m\le p-1$, define
\[R_p^{m}(\theta) = \begin{pmatrix}
    I_{m-1} & 0 & 0 & 0\\
    0 & \cos \theta & -\sin \theta & 0\\
    0 & \sin \theta & \cos \theta & 0\\
    0 & 0 & 0 & I_{p - m - 1}\\
\end{pmatrix}\]
to be the basic rotation matrix that sends a vector in the $p$-dimensional space to another by an anti-clockwise angle $0 \le \theta < 2\pi$ along the plane spanned by $e_m$ and $e_{m+1}$ via the mapping $v\mapsto R_p^m (\theta) \cdot v$, where $e_i = (0,\dots,1,\dots,0)'$ has only one $1$ in the $i$-th entry. Let 
\[\begin{aligned}
    H_p^1 (\underline{\theta_1}) = R_p^{p-1}(\theta_{1p})R_p^{p-2}(\theta_{1,p-1})R_p^{p-3}(\theta_{1,p-2})\dots R_p^{1}(\theta_{12}),\\
    H_p^2 (\underline{\theta_2}) = R_p^{p-1}(\theta_{2p})R_p^{p-2}(\theta_{2,p-1})\dots R_p^{2}(\theta_{23}),\\
    \vdots \qquad \\
    H_p^{p-1} (\underline{\theta_{p-1}}) = R_p^{p-1}(\theta_{p-1,p}).
\end{aligned}\]
Let
\(\mathcal D_p=\{\operatorname{diag}(\varepsilon_1,\ldots,\varepsilon_p):
\varepsilon_i\in\{-1,1\}\}\), and let
\(\mathcal R_p=O(p)/\mathcal D_p\). Every class in \(\mathcal R_p\)
has a representative in \(SO(p)\). Except on a null set, a representative
may be written as
\[
H(\underline{\underline\theta})
=H_p^1(\underline\theta_1)H_p^2(\underline\theta_2)\cdots
H_p^{p-1}(\underline\theta_{p-1}).
\]
The full \(SO(p)\) parametrization uses \(0\leq\theta_{ij}<\pi\) for
\(1\leq i<j<p\) and \(0\leq\theta_{ip}<2\pi\) for \(1\leq i<p\).
Identifying the \(2^{p-1}\) column sign choices with determinant one gives
a quotient fundamental domain. We use \(0\leq\theta_{ij}<\pi\) for all
\(1\leq i<j\leq p\) on this domain. 

Put \(c_d=\pi^{d^2/2}/\varGamma_d(d/2)\) for \(d\geq1\), and set
\(c_0=1\). Tumura \cite{Tumura1965THEDO} established

\begin{lemma}[Tumura's decomposition]\label{lem: Tumura} 
\begin{enumerate}
    \item The Haar measure on \(O(p)\) is given by
    \[(dH) = 2^p \prod_{i=1}^{p-2} \prod_{j=i+1}^{p-1} \sin^{p-j}\theta_{ij} (d\underline{\underline{\theta}}), \quad \text{where } 
        (d\underline{\underline{\theta}})
        =
        \bigwedge_{i=1}^{p-1}
        \bigwedge_{j=i+1}^{p}
        d\theta_{ij},\]
        with total mass 
        $ 2^p c_p.$
    \item The exterior volume element on the space of real symmetric matrices is given by
    \[(dS) = 2^{-p}\prod_{1\le i<j\le p} (\lambda_i-\lambda_j) (d\varLambda)(dH), \quad \text{where }(d\varLambda)=\bigwedge_{i=1}^{p}d\lambda_i.\]
\end{enumerate}
\end{lemma}

Since ${\rm vol} (O(p)) = 2^p c_p$, this lemma is useful.

\begin{lemma}
\label{lem:sugiyama}
Let \(N\geq1,r\in\mathbb N_0\), and let \(A,B\) be real symmetric \(N\times N\)
matrices. Then
\begin{equation}
\frac{1}{2^Nc_N}
\int_{O(N)}
\bigl[\operatorname{tr}(AHBH')\bigr]^r(dH)
=
\sum_{\substack{\mu\vdash r\\ \ell(\mu)\leq N}}
\frac{C_\mu(A)C_\mu(B)}{C_\mu(I_N)}.
\label{eq:orbital-moment-full}
\end{equation}
If \(B\) is diagonal, then the integrand descends to
\(\mathcal R_N=O(N)/\mathcal D_N\), and
\begin{equation}
\int_{\mathcal R_N}
\bigl[\operatorname{tr}(ARBR')\bigr]^r(dR)
=
c_N
\sum_{\substack{\mu\vdash r\\ \ell(\mu)\leq N}}
\frac{C_\mu(A)C_\mu(B)}{C_\mu(I_N)}.
\label{eq:orbital-moment-quotient}
\end{equation}
\end{lemma}

\begin{proof}
Equation~\eqref{eq:orbital-moment-full} is James's identity \cite[(28)-(29) in Theorem 5]{james1961distribution}, initially for positive definite \(A,B\), and
extends to arbitrary real symmetric matrices because both sides are
polynomials in their entries. If \(B\) is diagonal, then
\(DBD=B\) for every \(D\in \mathcal D_N\), so the integrand is right
\(\mathcal D_N\)-invariant. Since the quotient map
\(O(N)\to\mathcal R_N\) has \(2^N\) fibres,
\[
\int_{O(N)}f(H)(dH)
=
2^N\int_{\mathcal R_N}f(R)(dR).
\]
Combining this identity with
\eqref{eq:orbital-moment-full} gives
\eqref{eq:orbital-moment-quotient}.
\end{proof}

Sugiyama's \(k=1\) decomposition is the first step of the more
general block factorization used here. Let \(m=p-k\) and write
\begin{equation}
    H(\underline{\underline{\theta}})
    =
    H_k
    \begin{pmatrix}
    I_{k} & 0 \\
    0 & H_{m}
    \end{pmatrix},
    \quad
    H_k
    =
    H_p^1(\underline\theta_1)\cdots
    H_p^k(\underline\theta_k), \quad H_p^p = I_p.
    \label{eq:block-H}
\end{equation}
The first \(k\) columns of \(H_k\) are \(h_1,\ldots,h_k\). We also
set \(q_{0i}=h_i'\varSigma^{-1}h_i\) and define the nuisance block by
$\varSigma_m
=
(H_k'\varSigma^{-1}H_k)_{k+1:p,k+1:p}.$
The next section integrates out \(H_m\) and all \(p-k\) nuisance
eigenvalues.

\section{Joint Density of the Top
\texorpdfstring{\(k\)}{k} Principal Eigenvectors}
\label{sec:joint-density}

Fix \(1\leq k\leq p\). Let \(S\sim W_p(n,\varSigma)\), where
\(n>p-1\) and \(\varSigma\) is positive definite. Put
\(m=p-k\), \(v=(n-p-1)/2\), \(a=(n-k)/2\),
\(b=(p-k+3)/2\), and \(u=(p-k)(n-k)/2\). By
\cite{wishart1928generalised} and
\cite[Theorem~3.2.1]{muirhead1982aspects}, its probability element is
\begin{equation}
\frac{1}
{2^{np/2}\varGamma_p(n/2)|\varSigma|^{n/2}}
|S|^v
\operatorname{etr}
\left(
-\frac{1}{2}\varSigma^{-1}S
\right)
(dS),
\label{eq:wishart-density}
\end{equation}
When \(m=0\), use the conventions
\(B_0=1\), \(C_\varnothing(I_0)=1\), \(\varGamma_0=1\), and
\(\operatorname{tr}\varSigma_0=0\).
Write \(S=H\varLambda H'\), where
\(\varLambda=\operatorname{diag}(\lambda_1,\ldots,\lambda_p)\),
\(\lambda_1>\cdots>\lambda_p>0\), and \(H\) has the form
\eqref{eq:block-H}. Partition \(\varLambda\) as
\[
\varLambda=
\begin{pmatrix}
\varLambda_k & 0\\
0 & \varLambda_m
\end{pmatrix},
\]
Here \(\varLambda_k\) contains the leading \(k\) eigenvalues and
\(\varLambda_m\) contains the remaining \(m\) nuisance eigenvalues. To avoid
ambiguity between Lebesgue measure and the spectral Jacobian, write
\((d\lambda)=\bigwedge_{i=1}^p d\lambda_i\) and
\((d\omega_p)=\varDelta(\varLambda)(d\lambda)\).

By Tumura's decomposition on the column sign quotient, the
transformation from \(S\) to its ordered eigenvalues and eigenvectors converts
\eqref{eq:wishart-density} into
\begin{equation}
\begin{aligned}
&\frac{1}
{2^{np/2}\varGamma_p(n/2)|\varSigma|^{n/2}}
|\varLambda_k|^v|\varLambda_m|^v
\exp\left(
-\frac12\sum_{i=1}^k
\lambda_i h_i'\varSigma^{-1}h_i
\right)
\\
&\qquad\times
\operatorname{etr}
\left(
-\frac12\varSigma_mH_m\varLambda_mH_m'
\right)
(dH_m)(dH_k)(d\omega_p),
\end{aligned}
\label{eq:joint-eigendecomposition}
\end{equation}
where
\(\varSigma_m=(H_k'\varSigma^{-1}H_k)_{k+1:p,k+1:p}\).
The angular volume elements are
\begin{equation*}
\begin{aligned}
(dH_m)
&=
\prod_{i=k+1}^{p-2}
\prod_{j=i+1}^{p-1}
\sin^{p-j}\theta_{ij}
\bigwedge_{i=k+1}^{p-1}
\bigwedge_{j=i+1}^{p}
d\theta_{ij},
\\
(dH_k)
&=
\prod_{i=1}^{k}
\prod_{j=i+1}^{p-1}
\sin^{p-j}\theta_{ij}
\bigwedge_{i=1}^{k}
\bigwedge_{j=i+1}^{p}
d\theta_{ij}.
\end{aligned}
\end{equation*}

The derivation now proceeds in three stages. We first integrate out
the nuisance eigenvectors \(H_m\), then integrate the nuisance
eigenvalues \(\lambda_{k+1},\ldots,\lambda_p\), and finally evaluate
the remaining integral over the ordered top eigenvalues
\(\lambda_1,\ldots,\lambda_k\).

\subsection{Integration over the nuisance eigenvectors}

We first integrate over the angles \(\theta_{ij}\) with
\(i=k+1,\ldots,p-1\) and \(j=i+1,\ldots,p\), which are the variables
contained in \((dH_m)\). Applying
Lemma~\ref{lem:sugiyama} to the orthogonal integral in
\eqref{eq:joint-eigendecomposition}, we obtain the following marginal
probability element
\begin{equation}
\begin{aligned}
& \frac{\pi^{m^2/2}}
{2^{np/2}|\varSigma|^{n/2}
	\varGamma_p(n/2)\varGamma_m(m/2)}
\bigl(\lambda_1\cdots\lambda_k|\varLambda_m|\bigr)^v\\
& \qquad \times 
\exp\left(
-\frac12\sum_{i=1}^{k}
\lambda_i h_i'\varSigma^{-1}h_i
\right)\sum_{r=0}^{\infty}
\sum_{\substack{\mu\vdash r\\ \ell(\mu)\leq m}}
\frac{
	C_\mu(-\tfrac12\varSigma_m)
	C_\mu(\varLambda_m)
}{
	r!C_\mu(I_m)
}
(dH_k)(d\omega_p).
\end{aligned}
\label{eq:nuisance-angular}
\end{equation}

At this point the angular variables associated with the nuisance
eigenvectors have disappeared. The remaining difficulty lies in the
interaction between the leading and nuisance eigenvalues through the
Vandermonde factor.

\subsection{The mixed Vandermonde factor}

We next integrate over \(\lambda_{k+1},\ldots,\lambda_p\).
The principal obstacle is the Vandermonde product
\[
\varDelta
=
\prod_{1\leq i<j\leq p}
(\lambda_i-\lambda_j),
\]
which couples the nuisance eigenvalues simultaneously to the distinct
leading eigenvalues
\(\lambda_1,\ldots,\lambda_k\). Consequently, the nuisance eigenvalue
integral does not immediately reduce to a standard matrix beta
integral.

To isolate the mixed interaction, decompose the Vandermonde product as
\[
\varDelta
=
\underbrace{
\prod_{1\leq i<j\leq k}
(\lambda_i-\lambda_j)
}_{\varDelta_{11}}
\underbrace{
\prod_{k\leq i<j\leq p}
(\lambda_i-\lambda_j)
}_{\varDelta_{22}}
\underbrace{
\prod_{\substack{1\leq i< k<j\leq p}}
(\lambda_i-\lambda_j)
}_{\varDelta_{12}}.
\]
Here \(\varDelta_{11}\) depends only on the leading eigenvalues,
\(\varDelta_{22}\) contains \(\lambda_k\) together with the nuisance
eigenvalues, and \(\varDelta_{12}\) contains the remaining mixed
interaction.

Macdonald's dual summation identity in Lemma~\ref{lem: dual cauchy}
provides a finite expansion of \(\varDelta_{12}\)
\[
\begin{aligned}
\varDelta_{12}
&=
\prod_{i=1}^{k-1}
\prod_{j=k+1}^{p}
(\lambda_i-\lambda_j)
\\
&=
(\lambda_1\cdots\lambda_{k-1})^m
\prod_{i=1}^{k-1}
\prod_{j=k+1}^{p}
\left(
1-\lambda_i^{-1}\lambda_j
\right)
\\
&=
(\lambda_1\cdots\lambda_{k-1})^m
\sum_{s=0}^{m(k-1)}
\frac{1}{s!}
\sum_{\substack{
	\nu\vdash s, 
	\nu_1\leq k-1\\
	\ell(\nu)\leq m
	}}
	J_{\nu'}
	\left(
	-\varLambda_{k-1}^{-1};1/2
	\right)
	C_\nu(\varLambda_m).
\end{aligned}
\]
The upper bound \(m(k-1)\) is the degree of the mixed product. The
restrictions
\(\ell(\nu)\leq m\) and \(\nu_1\leq k-1\) reflect the numbers of
variables in \(\varLambda_m\) and
\(\varLambda_{k-1}^{-1}\), respectively.

Combining this expansion with the zonal polynomial factor already
present in \eqref{eq:nuisance-angular}, Lemma~\ref{lem: LR rule} gives
\[
C_\mu(\varLambda_m)C_\nu(\varLambda_m)
=
\sum_{\phi\vdash r+s}
g_{\mu\nu}^{\phi}C_\phi(\varLambda_m).
\]

It therefore remains to evaluate the nuisance eigenvalue integral
\[
\begin{aligned}
I_0
=
\int_{\lambda_k>\lambda_{k+1}>\cdots>\lambda_p>0}
|\varLambda_m|^v
\prod_{k\leq i<j\leq p}
(\lambda_i-\lambda_j)
C_\phi(\varLambda_m)
\bigwedge_{i=k+1}^{p}d\lambda_i.
\end{aligned}
\]
Introduce \(l_a=\lambda_{k+a}/\lambda_k\), \(a=1,\ldots,m\), and
\(L_m=\operatorname{diag}(l_1,\ldots,l_m)\).
Then
\[
\begin{aligned}
I_0
&=
\lambda_k^{u+m+r+s}
\int_{1>l_1>\cdots>l_m>0}
|L_m|^v
\prod_{a=1}^{m}(1-l_a)
\\
&\qquad\qquad\times
\prod_{1\leq a<b\leq m}(l_a-l_b)
C_\phi(L_m)
\bigwedge_{a=1}^{m}dl_a,
\end{aligned}
\]
Applying Lemma~\ref{lem: Selberg} with \(\alpha=2\) and the parameters
defined above yields
\[
I_0
=
\lambda_k^{u+m+r+s}
\frac{\varGamma_m(m/2)}{\pi^{m^2/2}}
\frac{(a)_\phi}{(a+b)_\phi}
B_m(a,b)C_\phi(I_m).
\]

For later use, define
\begin{equation}\label{eq:constants-Bmunu}
B_{\mu\nu}
=
\sum_{\substack{
	\phi\vdash|\mu|+|\nu|\\
	\ell(\phi)\leq m
	}}
	g_{\mu\nu}^{\phi}
	\frac{(a)_\phi}{(a+b)_\phi}
	C_\phi(I_m).
\end{equation}

Finally, define the remaining spectral measure by
\((d\omega_k)=\varDelta_{11}\bigwedge_{i=1}^k d\lambda_i\).
After combining the powers of the leading eigenvalues and cancelling
the angular and Selberg constants, equation \eqref{eq:nuisance-angular}
reduces to
\begin{equation}
\begin{aligned}
&\frac{B_m(a,b)}
{2^{np/2}\varGamma_p(n/2)|\varSigma|^{n/2}}
\exp\left(
-\frac12\sum_{i=1}^{k}
\lambda_i h_i'\varSigma^{-1}h_i
\right)
|\varLambda_k|^{v+m}
(d\omega_k)(dH_k)
\\
&\qquad\times
\sum_{r=0}^{\infty}
\sum_{\substack{\mu\vdash r\\ \ell(\mu)\leq m}}
\frac{
	C_\mu(-\tfrac12\varSigma_m)
}{
	r!C_\mu(I_m)
}
\sum_{s=0}^{m(k-1)}
\frac{1}{s!}
\sum_{\substack{
		\nu\vdash s,\nu_1\leq k-1\\
		\ell(\nu)\leq m}}
\lambda_k^{u+r+s}
B_{\mu\nu}
J_{\nu'}
\left(
-\varLambda_{k-1}^{-1};1/2
\right).
\end{aligned}
\label{eq:pre-final}
\end{equation}
Thus all nuisance variables have now been integrated out. The remaining
task is to integrate the top ordered eigenvalues over the Weyl
chamber.

\subsection{The final ordered integral}

Set \(q_{0i}=h_i'\varSigma^{-1}h_i\) for \(1\leq i\leq k\), and
define the Weyl
chamber $W_k=\{\lambda_1>\cdots>\lambda_k>0\}.$ 

The powers in \eqref{eq:pre-final} are not necessarily integers.
We therefore use the commuting Laplace multipliers from
Section~\ref{sec:fractional-operators}, for which
\(\mathscr D_i^ze^{-q_i\lambda_i/2}
=\lambda_i^ze^{-q_i\lambda_i/2}\) and
\(\mathscr D_i^{z_1}\mathscr D_i^{z_2}
=\mathscr D_i^{z_1+z_2}\).

To evaluate the basic ordered exponential integral, introduce the
spacing variables \(t_i=\lambda_i-\lambda_{i+1}\) for
\(1\leq i<k\) and \(t_k=\lambda_k\).
Then
\[
\lambda_i=t_i+\cdots+t_k,
\quad
\sum_{i=1}^{k}q_i\lambda_i
=
\sum_{j=1}^{k}(q_1+\cdots+q_j)t_j.
\]
Consequently,
\[
\varPhi_k(\underline q)
=
\int_{W_k}
e^{-\frac12\sum_{i=1}^{k}q_i\lambda_i}
(d\underline\lambda)
=
\frac{2^k}
{q_1(q_1+q_2)\cdots(q_1+\cdots+q_k)}.
\]

More generally, multiplication by arbitrary powers of the eigenvalues
can be transferred to the Laplace variables
\begin{equation}
\int_{W_k}
\prod_{i=1}^{k}\lambda_i^{z_i}
e^{-\frac12\sum_{i=1}^{k}q_i\lambda_i}
(d\underline\lambda)
=
\left(
\prod_{i=1}^{k}\mathscr D_i^{z_i}
\right)
\varPhi_k(\underline q).
\label{eq:frac}
\end{equation}

Fix a partition \(\nu\vdash s\), and put
\(d=k-1\),
\(\varLambda_d=\operatorname{diag}(\lambda_1,\ldots,\lambda_d)\), and
\(\varDelta_d=\prod_{1\leq i<j\leq d}(\lambda_i-\lambda_j)\).
Partitions are padded with trailing zeros whenever an indexed part is
required. The determinant of order \(d=0\) and all products indexed by
\(1,\ldots,d\) are one.
Expand the Jack polynomial in the Schur basis
\begin{equation}
J_{\nu'}
\left(
-\varLambda_d^{-1};1/2
\right)
=
\sum_{\substack{
	\sigma\vdash s, \sigma\leq \nu'\\
	\ell(\sigma)\leq k-1
	}}
	K_{\nu'\sigma}(1/2)
	s_\sigma
	\left(
	-\varLambda_d^{-1}
	\right).\label{eq:Kostka-1-2}
\end{equation}

By the Weyl character formula \eqref{lem:Weyl},
\[
\varDelta_d
s_\sigma(-\varLambda_d^{-1})
=
(-1)^s
\det\left(
\lambda_i^{d-j-\sigma_{d+1-j}}
\right)_{i,j=1}^{d}.
\]
Since
\(
\varDelta_{11}
=
\varDelta_d
\prod_{i=1}^{d}
(\lambda_i-\lambda_k)\), the product
\(\varDelta_{11}
|\varLambda_k|^{v+m}
\lambda_k^{u+r+s}
s_\sigma(-\varLambda_d^{-1})
\)
can be represented by a single determinant whose entries are
products of powers of the ordered eigenvalues. The columns are written in
reverse partition order. This absorbs the fixed orientation factor
\((-1)^{\binom d2}=(-1)^{\binom{k-1}{2}}\) from the Weyl determinant.

Since \(|\sigma|=s\), multiplying the column indexed by
\(\sigma_{d+1-j}\) by \((-1)^{\sigma_{d+1-j}}\) also incorporates
the factor \((-1)^s\) in the displayed identity. In the explicit
determinants below, the first two and final rows and columns are shown.
For \(d=0\) or \(d=1\), the displayed pattern is truncated to the
corresponding order.

Applying \eqref{eq:frac} term by term gives the compact operator
definition
\begin{equation}
\begin{aligned}
\mathscr M_{\sigma,r,s}^{(k)}(\underline q)
&:=
\mathscr D_k^{v+u+m+r+s}
\prod_{i=1}^{d}
\left\{(\mathscr D_i-\mathscr D_k)\mathscr D_i^{v+m}\right\}
\\
&\qquad\times
(-1)^s\det\left(
\mathscr D_i^{d-j-\sigma_{d+1-j}}
\right)_{i,j=1}^{d}
\varPhi_k(\underline q).
\end{aligned}
\label{eq:calM}
\end{equation}

Substituting \eqref{eq:calM} into \eqref{eq:pre-final} gives a first
series representation of the joint density. Its direct termwise
derivation requires a convergence condition that is not automatic.
\begin{theorem}[Local series representation]
\label{thm:joint-density-top-k}
Let \(1\leq k\leq p\), \(n>p-1\), and let \(\varSigma\) be positive
definite. Put \(d=k-1\), \(m=p-k\), \(v=(n-p-1)/2\), \(a=(n-k)/2\),
\(b=(p-k+3)/2\), and \(u=(p-k)(n-k)/2\). If
\begin{equation}
\sum_{i=1}^k q_{0i}>\operatorname{tr}\varSigma_m,
\label{eq:local-convergence-condition}
\end{equation}
then the following series converges absolutely and is the joint
probability element of \(h_1,\ldots,h_k\)
\begin{align}
&\frac{B_m(a,b)}
{2^{np/2}\varGamma_p(n/2)|\varSigma|^{n/2}}
\sum_{r=0}^{\infty}
\sum_{\substack{\mu\vdash r\\\ell(\mu)\leq m}}
\frac{C_\mu(-\tfrac12\varSigma_m)}{r!C_\mu(I_m)}
\sum_{s=0}^{m(k-1)}\frac{1}{s!}
\sum_{\substack{\nu\vdash s,\ \nu_1\leq k-1\\
\ell(\nu)\leq m}}
B_{\mu\nu}
\sum_{\substack{\sigma\vdash s,\ \sigma\leq\nu'\\
\ell(\sigma)\leq k-1}}
\notag\\
&\qquad\times
K_{\nu'\sigma}(1/2)\prod_{i=1}^k\prod_{j=i+1}^{p-1}
\sin^{p-j}\theta_{ij}
\bigwedge_{i=1}^k\bigwedge_{j=i+1}^{p}d\theta_{ij}\times
\mathscr D_k^{v+u+m+r+s}
\prod_{i=1}^{d}(\mathscr D_i-\mathscr D_k)\notag\\
&\qquad\times\bigl(\prod_{i=1}^{d}\mathscr D_i^{v+m}\bigr)
\begin{vmatrix}
(-1)^{\sigma_d}\mathscr D_1^{d-1-\sigma_d}
& (-1)^{\sigma_{d-1}}\mathscr D_1^{d-2-\sigma_{d-1}}
& \cdots
& (-1)^{\sigma_1}\mathscr D_1^{-\sigma_1}\\
(-1)^{\sigma_d}\mathscr D_2^{d-1-\sigma_d}
& (-1)^{\sigma_{d-1}}\mathscr D_2^{d-2-\sigma_{d-1}}
& \cdots
& (-1)^{\sigma_1}\mathscr D_2^{-\sigma_1}\\
\vdots & \vdots & \ddots & \vdots\\
(-1)^{\sigma_d}\mathscr D_d^{d-1-\sigma_d}
& (-1)^{\sigma_{d-1}}\mathscr D_d^{d-2-\sigma_{d-1}}
& \cdots
& (-1)^{\sigma_1}\mathscr D_d^{-\sigma_1}
\end{vmatrix}
\varPhi_k(\underline q)
\Big|_{\underline q=\underline q_0}.
\label{eq:final-local}
\end{align}
Here \(\nu'\) denotes the conjugate partition, while
\(K_{\nu'\sigma}(1/2)\) and \(B_{\mu\nu}\) are defined in
\eqref{eq:Kostka-1-2} and \eqref{eq:constants-Bmunu}, respectively.
\end{theorem}

\begin{proof}
The case \(m=0\) contains only the empty partition, so assume
\(m\geq1\). Fix the leading frame and put
\(A=\varSigma_m\), \(Q_j=\sum_{i=1}^jq_{0i}\), and
\(\delta=Q_k-\operatorname{tr}A>0\). For \(\mu\vdash r\), homogeneity
gives \(C_\mu(-A/2)=(-1)^rC_\mu(A/2)\). Because \(0<\varLambda_m<\lambda_kI_m\), the quotient form of Lemma~\ref{lem:sugiyama}, applied with the matrix arguments \(A/2\) and \(\varLambda_m\), gives
\[
\begin{aligned}
&c_m
\sum_{r=0}^{\infty}
\sum_{\substack{\mu\vdash r\\\ell(\mu)\leq m}}
\frac{|C_\mu(-A/2)|C_\mu(\varLambda_m)}
{r!C_\mu(I_m)}
\\
&\qquad=
\int_{\mathcal R_m}
\exp\left\{
\frac12\operatorname{tr}(AR\varLambda_mR')
\right\}(dR)
\leq
c_m\exp\left\{
\frac{\lambda_k}{2}\operatorname{tr}A
\right\}.
\end{aligned}
\]

The dual Cauchy expansion is finite. Moreover,
\(\nu'_1=\ell(\nu)\leq m\) and \(\sigma_1\leq\nu'_1\), so the factor
\((\lambda_1\cdots\lambda_d)^m\) removes every negative power in the
Jack and Schur terms. Thus the absolute values of the remaining finite
terms are bounded by a polynomial in the eigenvalues. Up to such a
polynomial factor, the absolute integrand is bounded by
\[
\prod_{i=1}^p\lambda_i^v
\exp\left[-\frac12\left\{
\sum_{i=1}^kq_{0i}\lambda_i
-\lambda_k\operatorname{tr}A
\right\}\right].
\]
Let \(t_j=\lambda_j-\lambda_{j+1}\) for \(j<k\), and let
\(t_k=\lambda_k\). Then
\[
\sum_{i=1}^kq_{0i}\lambda_i
-\lambda_k\operatorname{tr}A
=
\sum_{j=1}^{k-1}Q_jt_j+\delta t_k.
\]
All coefficients on the right are positive. Also,
\(\sum_{i=1}^p\lambda_i\leq
\sum_{j=1}^{k-1}jt_j+pt_k\), so the exponential dominates every
polynomial at infinity. At the zero and collision boundaries,
\(v>-1\) and the Vandermonde factors are locally integrable. The
termwise absolute integral is therefore finite.

Tonelli's theorem now justifies interchanging the zonal series with the
orthogonal and eigenvalue integrals. All subsequent partition
expansions are finite. Equation~\eqref{eq:frac} applies term by term
because, for \(\ell=d+1-j\),
\[
v+m+d-j-\sigma_{d+1-j}
=v+m+\ell-1-\sigma_\ell
\geq v+\ell-1>-1,
\]
and \(v+u+m+r+s>-1\). The ordered Selberg integral contributes
\(c_m^{-1}B_m(a,b)\), which cancels the factor \(c_m\) from the
orthogonal integral. Finally, the reverse column order absorbs
\((-1)^{\binom d2}\), and the column signs multiply to
\((-1)^{|\sigma|}=(-1)^s\). This gives \eqref{eq:final-local} with the
stated constant and proves absolute convergence.
\end{proof}

For \(k=p\), all zero-dimensional quantities are interpreted by the
conventions \(B_0=1\), \(C_\varnothing(I_0)=1\), and
\(\varGamma_0=1\). Sums over partitions of length zero then contain
only the empty partition.

For \(k=1\), all determinants and products indexed by
\(1,\ldots,k-1\) are empty and equal \(1\). Moreover,
\(s=0\), \(\nu=\sigma=\varnothing\), and
\(\varPhi_1(q_1)=2/q_1\). Therefore,
\[
\mathscr M_{\varnothing,u+r}^{(1)}(q_1)
=
\mathscr D_1^{v+u+r}\frac{2}{q_1}
=
\frac{
2^{v+u+r+1}
\varGamma(v+u+r+1)
}{
q_1^{v+u+r+1}
},
\]
which recovers the single gamma integral in
Sugiyama's formula \eqref{eq:sugiyama66}.

\subsection{Applying Kummer transformation}
\label{sec:sugiyama-revisited}

Condition \eqref{eq:local-convergence-condition} enters when the
orthogonal exponential is expanded before the unbounded leading eigenvalue
integral is taken. 
Also, \eqref{eq:final-local} is an alternating series in $-\tfrac12 \varSigma_m$. Thus the underlying
density may exist for every positive definite \(\varSigma\), but the
particular series in \eqref{eq:final-local} is not a globally valid
termwise representation.

The issue is already visible in the smallest case \(k=1\) and \(p=2\).
Let \(h_2\) be a unit vector orthogonal to \(h_1\), and set
\(q_1=h_1'\varSigma^{-1}h_1\),
\(q_2=h_2'\varSigma^{-1}h_2\), \(a=(n-1)/2\), and
\(c=(n+3)/2\).
The specialization of \eqref{eq:final-local}, with respect to the
angular element \(d\theta_{12}\), is
\begin{equation}
\frac{B_1(a,2)\varGamma(n)}
{\varGamma_2(n/2)|\varSigma|^{n/2}}
q_1^{-n}
{}_2F_1\left(n,a;c;-\frac{q_2}{q_1}\right)d\theta_{12}.
\label{eq:sugiyama-p2-local}
\end{equation}
The defining Gauss series in \eqref{eq:sugiyama-p2-local} is absolutely
convergent for \(q_2/q_1<1\). Kummer's form of Pfaff's transformation
gives
\[
q_1^{-n}{}_2F_1\left(n,a;c;-\frac{q_2}{q_1}\right)
=
(q_1+q_2)^{-n}{}_2F_1\left(n,2;c;\frac{q_2}{q_1+q_2}\right).
\]
Since \(0<q_2/(q_1+q_2)<1\), the series on the right converges throughout
the positive definite parameter space. This is the scalar instance of
the complement transformation used below and agrees with the
globally convergent version of Sugiyama's formula \cite{Sugiyama1966,Ishizaki2009}.

For the general case, define \(\tilde q_i=q_i\) and
\(\tilde q_{0i}=q_{0i}\) for \(1\leq i<k\), together with
\(\tilde q_k=q_k+\operatorname{tr}\varSigma_m\), and
\(\tilde q_{0k}=q_{0k}+\operatorname{tr}\varSigma_m\). Also, define
\begin{equation}\label{eq:constants-Bmunu-global}
\widetilde B_{\mu\nu}
=
\sum_{\substack{
	\phi\vdash|\mu|+|\nu|\\
	\ell(\phi)\leq m
	}}
	g_{\mu\nu}^{\phi}
	\frac{(b)_\phi}{(a+b)_\phi}
	C_\phi(I_m).
\end{equation}
Using the same reverse partition order, define
\begin{equation}
\begin{aligned}
\widetilde{\mathscr M}_{\sigma,r,s}^{(k)}(\underline{\tilde q})
=
\mathscr D_k^{v+u+m+r+s}
\prod_{i=1}^{d}\mathscr D_i^v
\det\left(
(\mathscr D_i-\mathscr D_k)^{m+d+1-j-\sigma_{d+1-j}}
\right)_{i,j=1}^{d}
\varPhi_k(\underline{\tilde q}).
\end{aligned}
\label{eq:calM-global}
\end{equation}

\begin{theorem}[Globally convergent series representation]
\label{cor:joint-density-global}
Let \(1\leq k\leq p\), \(n>p-1\), and let \(\varSigma\) be positive
definite. Put \(d=k-1\), \(m=p-k\), \(v=(n-p-1)/2\), \(a=(n-k)/2\),
\(b=(p-k+3)/2\), and \(u=(p-k)(n-k)/2\). The joint probability element of
\(h_1,\ldots,h_k\) is
\begin{align}
&\frac{B_m(a,b)}
{2^{np/2}\varGamma_p(n/2)|\varSigma|^{n/2}}
\sum_{r=0}^{\infty}
\sum_{\substack{\mu\vdash r\\\ell(\mu)\leq m}}
\frac{C_\mu(\tfrac12\varSigma_m)}{r!C_\mu(I_m)}
\sum_{s=0}^{m(k-1)}\frac{1}{s!}
\sum_{\substack{\nu\vdash s,\ \nu_1\leq k-1\\
\ell(\nu)\leq m}} \widetilde B_{\mu\nu}
\sum_{\substack{\sigma\vdash s,\ \sigma\leq\nu'\\
\ell(\sigma)\leq k-1}}
\notag\\
&\qquad\times
K_{\nu'\sigma}(1/2) \prod_{i=1}^k\prod_{j=i+1}^{p-1}
\sin^{p-j}\theta_{ij}
\bigwedge_{i=1}^k\bigwedge_{j=i+1}^{p}d\theta_{ij}\times 
\mathscr D_k^{v+u+m+r+s}
\prod_{i=1}^{d}\mathscr D_i^v
\notag\\
&\qquad\times
\begin{vmatrix}
(\mathscr D_1-\mathscr D_k)^{m+d-\sigma_d}
& (\mathscr D_1-\mathscr D_k)^{m+d-1-\sigma_{d-1}}
& \cdots
& (\mathscr D_1-\mathscr D_k)^{m+1-\sigma_1}\\
(\mathscr D_2-\mathscr D_k)^{m+d-\sigma_d}
& (\mathscr D_2-\mathscr D_k)^{m+d-1-\sigma_{d-1}}
& \cdots
& (\mathscr D_2-\mathscr D_k)^{m+1-\sigma_1}\\
\vdots & \vdots & \ddots & \vdots\\
(\mathscr D_d-\mathscr D_k)^{m+d-\sigma_d}
& (\mathscr D_d-\mathscr D_k)^{m+d-1-\sigma_{d-1}}
& \cdots
& (\mathscr D_d-\mathscr D_k)^{m+1-\sigma_1}
\end{vmatrix}
\varPhi_k(\underline{\tilde q})
\Big|_{\underline{\tilde q}=\underline{\tilde q}_0}.
\label{eq:final-global}
\end{align}
The infinite series is absolutely convergent over the entire positive
definite matrix space.
\end{theorem}

\begin{proof}
The case \(m=0\) is immediate, so assume \(m\geq1\) and put
\(A=\varSigma_m\). In the nuisance eigenvalue integral, set
\(y_a=\lambda_k-\lambda_{p+1-a}\), \(a=1,\ldots,m\). By Haar
invariance, the orthogonal exponential becomes
\[
\exp\left\{-\frac{\lambda_k}{2}\operatorname{tr}A\right\}
\operatorname{etr}\left\{
\frac12AR\operatorname{diag}(y_1,\ldots,y_m)R'
\right\}.
\]
Hence \(q_k\) is replaced by
\(\widetilde q_k=q_k+\operatorname{tr}A\), and
\(\sum_{i=1}^k\widetilde q_{0i}=\operatorname{tr}\varSigma^{-1}\).

After scaling \(y_a=\lambda_kx_a\), the nuisance spectral measure is
\(\lambda_k^{u+m}|X||I_m-X|^v\Delta(X)\,dX\). Since
\(1=b-(m+1)/2\) and \(v=a-(m+1)/2\), this is the ordered matrix beta
weight with parameters \((b,a)\). The positive orbital expansion and
Macdonald's dual identity contribute respectively the factors indexed
by \(\mu\) and \(\nu\); the Selberg--Jack integral then gives
\(c_m^{-1}B_m(a,b)\widetilde B_{\mu\nu}\), cancelling the orbital
factor \(c_m\).

Put \(\delta_i=\lambda_i-\lambda_k\). The mixed Vandermonde factors as
\[
\prod_{i=1}^{d}\delta_i^m
\prod_{i=1}^{d}\prod_{a=1}^{m}
\left(1+\frac{\lambda_kx_a}{\delta_i}\right).
\]
After the Jack--Schur expansion, the reverse-column Weyl formula turns
\(\Delta_{11}\prod_i\delta_i^m\) into the determinant with entries
\(\delta_i^{m+d+1-j-\sigma_{d+1-j}}\). The substitutions
\(\lambda_i\mapsto\mathscr D_i\) and
\(\delta_i\mapsto\mathscr D_i-\mathscr D_k\) therefore yield
\eqref{eq:calM-global}, and hence \eqref{eq:final-global}.

Finally, for \(0<X<I_m\),
\[
\exp\left\{-\frac{\lambda_k}{2}\operatorname{tr}A\right\}
\operatorname{etr}\left\{\frac{\lambda_k}{2}ARXR'\right\}
=
\exp\left\{-\frac{\lambda_k}{2}
\operatorname{tr}[A(I_m-RXR')]\right\}\leq1.
\]
Thus the homogeneous orbital expansion has nonnegative terms and is
dominated by the original Wishart kernel, so Tonelli's theorem applies
for every \(A>0\). The remaining Macdonald and Jack--Schur sums are
finite, and the powers of \(\delta_i\) cancel all reciprocal factors.
Therefore the resulting series is absolutely convergent throughout
the positive definite matrix space.
\end{proof}

For \(k=1\), Theorem~\ref{cor:joint-density-global} gives the
Kummer-transformed Sugiyama series for arbitrary \(p\).
For \(k=2\), Section~\ref{sec:k2-explicit} gives a Gauss
hypergeometric representation of the globally convergent terms, and
Appendix~\ref{app:k2} records the specializations for \(p=3,4,5\).
The coefficients \(K_{\lambda\mu}(1/2)\) through degree six are listed
in Appendix~\ref{app:kostka}.

\section{Explicit Evaluation for \texorpdfstring{\(k=2\)}{k=2}}
\label{sec:k2-explicit}

We now specialize the globally convergent representation in
Theorem~\ref{cor:joint-density-global}. Put
\(m=p-2\), \(v=(n-p-1)/2\), \(u=(p-2)(n-2)/2\),
\(\tilde q_2=q_2+\operatorname{tr}\varSigma_m\), and
\(\widetilde Q=q_1+\tilde q_2\).
For \(r\in\mathbb N_0\) and \(0\leq s\leq m\), define
\[
\begin{aligned}
\widetilde{\mathcal M}_{r,s}^{(p)}(q_1,\tilde q_2)
&=
\int_{\lambda_1>\lambda_2>0}
\lambda_1^v
\lambda_2^{v+u+m+r+s}
(\lambda_1-\lambda_2)^{m+1-s}
\\
&\qquad\times
\exp\left\{-\frac12
(q_1\lambda_1+\tilde q_2\lambda_2)\right\}
d\lambda_1\,d\lambda_2.
\end{aligned}
\]
This is the \(k=2\) specialization of
\(\widetilde{\mathscr M}_{(s),r,s}^{(2)}\). Since
\(K_{(s),(s)}(1/2)=(s+1)!/2^s\), the finite partition sum in the global
density is
\begin{equation}
\sum_{s=0}^{m}
\frac{s+1}{2^s}\,
\widetilde B_{\mu,(1^s)}
\widetilde{\mathcal M}_{r,s}^{(p)}(q_1,\tilde q_2).
\label{eq:k2-global-finite-sum}
\end{equation}
Every term in \eqref{eq:k2-global-finite-sum} is nonnegative.

\begin{theorem}[Stable Gauss hypergeometric representation]
\label{thm:k2-hypergeometric}
Let \(A_r=np/2+r\) and
\(\tilde\beta_{r,s}=v+u+m+r+s\). Also put
\(\tilde c_r=v+u+2m+r+3\).
For \(q_1,\tilde q_2>0\), \(n>p-1\), \(r\in\mathbb N_0\), and
\(0\leq s\leq p-2\),
\begin{align}
\widetilde{\mathcal M}_{r,s}^{(p)}
&=
2^{A_r}\varGamma(A_r)
B\bigl(\tilde\beta_{r,s}+1,m+2-s\bigr)
q_1^{-A_r}
{}_2F_1\left(
A_r,\tilde\beta_{r,s}+1;
\tilde c_r;
-\frac{\tilde q_2}{q_1}
\right)
\notag\\
&=
2^{A_r}\varGamma(A_r)
B\bigl(\tilde\beta_{r,s}+1,m+2-s\bigr)
\widetilde Q^{-A_r}
{}_2F_1\left(
A_r,m+2-s;
\tilde c_r;
\frac{\tilde q_2}{\widetilde Q}
\right).
\label{eq:k2-global-hypergeometric}
\end{align}
The second line has argument in \((0,1)\) and is therefore the stable
form for numerical evaluation. The first line is understood through
analytic continuation when \(\tilde q_2/q_1\geq1\).
\end{theorem}

\begin{proof}
The substitution \(\lambda_1=x\) and \(\lambda_2=tx\) gives
\[
\widetilde{\mathcal M}_{r,s}^{(p)}
=
2^{A_r}\varGamma(A_r)
\int_0^1
\frac{
t^{\tilde\beta_{r,s}}
(1-t)^{m+1-s}
}{
(q_1+\tilde q_2t)^{A_r}
}\,dt.
\]
Euler's integral yields the first line of
\eqref{eq:k2-global-hypergeometric}, and the Kummer--Pfaff
transformation yields the second. In conjunction with the positive
coefficients \(C_\mu(\varSigma_m/2)\) and
\(\widetilde B_{\mu,(1^s)}\), this representation preserves the global
absolute convergence established in
Theorem~\ref{cor:joint-density-global}.
\end{proof}

\begin{corollary}[Elementary form when \(n=p+1\)]
\label{cor:k2-elementary}
If \(n=p+1\), then \(v=0\). With
\(\tilde\beta_{r,s}=u+m+r+s\),
\[
\widetilde{\mathcal M}_{r,s}^{(p)}
=
\frac{
2^{A_r}
\varGamma(\tilde\beta_{r,s}+1)
\varGamma(m+2-s)
}{
\widetilde Q^{\,\tilde\beta_{r,s}+1}
q_1^{m+2-s}
}.
\]
\end{corollary}

\begin{proof}
The formula follows directly from the spacing variables
\(x=\lambda_2\) and \(y=\lambda_1-\lambda_2\), because the factor
\(\lambda_1^v\) equals one.
\end{proof}

Appendix~\ref{app:k2} records the
reduction and the low-dimensional cases \(p=3,4,5\).
\section{The Pushforward Law and Exact Principal Subspace Inference}
\label{sec:principal-subspace}

We push the ordered frame law in
Theorem~\ref{cor:joint-density-global} onto the Grassmannian and use
the resulting exact finite-sample law to assess the Gaussian weighted
bootstrap of \cite{NaumovSpokoinyUlyanov2019}.

\subsection{The pushforward Grassmann density}

Let \[\mathbb V_{p,k}=\{U\in\mathbb R^{p\times k}:U'U=I_k\}, \quad \mathbb G(p,k)=\{P=P'=P^2:\operatorname{tr}P=k\}.\]
The eigenvector object is a frame of ordered unoriented lines,
\(\mathbb V_{p,k}^{\pm}=\mathbb V_{p,k}/\mathcal D_k\). We continue
to write \(U\) for one of its representatives. The natural projection
maps \([U]\) to \(UU'\). Each fibre is naturally identified with
\(\mathcal R_k=O(k)/\mathcal D_k\).

Use the quotient measures of Section~\ref{sec:orthogonal}, so that
\(\int_{\mathcal R_k}(dR)=c_k\). Let \((dP)\) be the invariant measure
characterized by
\begin{equation}
\int_{\mathbb V_{p,k}^{\pm}}f(U)(dH_k)
=
\int_{\mathbb G(p,k)}\int_{\mathcal R_k}f(UR)(dR)(dP),
\label{eq:stiefel-grassmann-disintegration}
\end{equation}
for every integrable function \(f\) invariant under column sign
changes. Retain the notation \(m=p-k\), \(a\), \(b\), \(u\), \(v\), and
\(\widetilde B_{\mu\nu}\) from
Theorem~\ref{cor:joint-density-global}, and write its prefactor as
\(\mathcal N_{p,k,n}(\varSigma)=B_m(a,b)/
\{2^{np/2}\varGamma_p(n/2)|\varSigma|^{n/2}\}\).
The total mass of \((dP)\) is
\[
\operatorname{vol}\mathbb G(p,k)
=
\frac{c_p}{c_kc_m}
=
\pi^{km/2}
\frac{\varGamma_k(k/2)}{\varGamma_k(p/2)}.
\]

For \(P\in\mathbb G(p,k)\), choose \(U_P\in\mathbb V_{p,k}\) with
\(U_PU_P'=P\), choose an orthonormal basis matrix \(U_{P,\perp}\) for
its orthogonal complement, and put
\begin{equation}
\begin{aligned}
\mathcal A(P) & =U_P'\varSigma^{-1}U_P,\\
\mathcal A_\perp(P) & =U_{P,\perp}'\varSigma^{-1}U_{P,\perp},\\
q_i(P,R) & =e_i'R'\mathcal A(P)Re_i.
\label{eq:section8-compressions}
\end{aligned}
\end{equation}
Set \(\tilde q_i(P,R)=q_i(P,R)\) for \(1\leq i<k\) and
\(\tilde q_k(P,R)=q_k(P,R)+
\operatorname{tr}\mathcal A_\perp(P)\).

\begin{theorem}[The pushforward Grassmann density]
\label{thm:grassmann-pushforward}
Let \(1\leq k\leq p\), let \(S\sim W_p(n,\varSigma)\), where
\(n>p-1\) and \(\varSigma\) is positive definite, and put
\(\widehat P_k=\sum_{i=1}^kh_ih_i'\). Relative to the invariant
quotient measure \((dP)\) in
\eqref{eq:stiefel-grassmann-disintegration}, the exact density of
\(\widehat P_k\) is
\begin{align}
g_{p,k}(P,\varSigma)
&=
\mathcal N_{p,k,n}(\varSigma)
\sum_{r=0}^{\infty}
\sum_{\substack{\mu\vdash r\\\ell(\mu)\leq m}}
\frac{
	C_\mu(\tfrac12\mathcal A_\perp(P))
}{r!C_\mu(I_m)}
\sum_{s=0}^{m(k-1)}\frac{1}{s!}
\sum_{\substack{\nu\vdash s,\ \nu_1\leq k-1\\
		\ell(\nu)\leq m}}
\widetilde B_{\mu\nu}
\notag\\
&\qquad\times
\sum_{\substack{\sigma\vdash s,\ \sigma\leq\nu'\\
		\ell(\sigma)\leq k-1}}
K_{\nu'\sigma}(1/2)
\int_{\mathcal R_k}
\widetilde{\mathscr M}_{\sigma,r,s}^{(k)}
(\underline{\tilde q}(P,R))(dR).
\label{eq:grassmann-density}
\end{align}
The right hand side is independent of the choices of \(U_P\) and
\(U_{P,\perp}\). 
\end{theorem}

\begin{proof}
Integrating the ordered frame density in
Theorem~\ref{cor:joint-density-global} along the fibres in
\eqref{eq:stiefel-grassmann-disintegration} yields
\eqref{eq:grassmann-density}. Replacing \(U_P\) by \(U_PQ\) changes the
fibre variable from \(R\) to \(Q'R\), so invariance of \((dR)\) proves
independence from the chosen frame. Changing \(U_{P,\perp}\) conjugates
\(\mathcal A_\perp(P)\), which leaves every zonal polynomial unchanged.
Finally, integration over \(P\) gives one because the frame law is
normalized.
\end{proof}

\subsection{Identifiability and the rank-two testing problem}

For the remainder of this section, fix \(p=4\) and \(k=2\). Let
\(X_1,\ldots,X_n\) be independent with
\(X_i\sim N_4(0,\varSigma)\), where \(n>3\), and set
\[
\widehat{\varSigma}=\frac1n\sum_{i=1}^nX_iX_i'.
\]
Let \(\widehat P_2\) be the projector onto the two leading sample
eigendirections. We consider the two-level covariance model
\begin{equation}
\varSigma(\lambda_2,\rho,P_*)
=
\lambda_1P_*+\lambda_2(I_4-P_*)
=
\lambda_2\{I_4+(\rho-1)P_*\},
\quad
\rho=\frac{\lambda_1}{\lambda_2}>1,
\label{eq:two-level-covariance}
\end{equation}
where \(P_*\in\mathbb G(4,2)\) and \(\lambda_2>0\). The leading and
trailing population eigenvalues both have multiplicity two. Thus no
particular orthonormal basis within either eigenspace is identifiable.
The leading subspace itself is identifiable because the two blocks are
separated by the eigengap \(\lambda_1-\lambda_2>0\). Indeed,
\[
P_*=\frac{\varSigma-\lambda_2I_4}{\lambda_1-\lambda_2}.
\]
At \(\rho=1\), this expression is undefined and the distribution is
isotropic, so \(P_*\) is not identifiable. The regime
\(\rho\downarrow1\) is consequently a weak identification boundary.
By contrast, the eigenvalues of \(\widehat{\varSigma}\) are distinct
with probability one, and \(\widehat P_2\) is well defined.

For \(P_0\in\mathbb G(4,2)\) specified independently of the sample,
consider
\[
H_0:P_*=P_0
\quad\text{against}\quad
H_1:P_*\ne P_0.
\]
Let \(0\leq\theta_1\leq\theta_2\leq\pi/2\) be the principal angles
between \(\widehat P_2\) and \(P_0\). The squared chordal distance on
\(\mathbb G(4,2)\) is
\[
\mathfrak d^2(P,Q)
=
\sum_{j=1}^2\sin^2\theta_j(P,Q)
=
\frac12\|P-Q\|_F^2.
\]
We use the statistic
\[
\begin{aligned}
T_{n,2}(P_0)
&=
n\|\widehat P_2-P_0\|_F^2
=
2n\{2-\operatorname{tr}(\widehat P_2P_0)\}\\
&=
2n\sum_{j=1}^2\sin^2\theta_j
=
2n\,\mathfrak d^2(\widehat P_2,P_0).
\end{aligned}
\]
It is invariant under all choices of bases within the two subspaces and
takes values in \([0,4n]\). Since multiplication of
\(\widehat{\varSigma}\) by a positive scalar does not change its
eigenvectors, the null distribution of \(T_{n,2}(P_0)\) is independent
of \(\lambda_2\). It depends on the covariance model only through
\(n\) and \(\rho\).

\subsection{The exact null distribution and confidence regions}

Under \(H_0\), orthogonal invariance allows us to take
\(P_0=\operatorname{diag}(I_2,0)\) and \(\lambda_2=1\), so that
\(\varSigma=\operatorname{diag}(\rho I_2,I_2)\). Let
\(1>x_1>x_2>0\) be the squared cosines of the two principal angles
between \(P\) and \(P_0\), and put \(D_x=\operatorname{diag}(x_1,x_2)\).
The \(O(2)\times O(2)\) symmetry permits the representatives
\[
U_x=
\begin{pmatrix}D_x^{1/2}\\(I_2-D_x)^{1/2}\end{pmatrix},
\quad
U_{x,\perp}=
\begin{pmatrix}-(I_2-D_x)^{1/2}\\D_x^{1/2}\end{pmatrix}.
\]
With \(R_\psi\) denoting a rotation in two dimensions, the compressions
in \eqref{eq:section8-compressions} become
\[
\begin{aligned}
\mathcal A_\perp(x)
&=
\operatorname{diag}\{\rho^{-1}+(1-\rho^{-1})x_i\}_{i=1}^2,\\
\zeta_i
&=1-(1-\rho^{-1})x_i,
\quad i=1,2,\\
q_1(x,\psi)
&=
\zeta_1\cos^2\psi+\zeta_2\sin^2\psi,\\
q_2(x,\psi)
&=
\zeta_1\sin^2\psi+\zeta_2\cos^2\psi,\\
\tau(x)
&=
\operatorname{tr}\mathcal A_\perp(x)
=2\rho^{-1}+(1-\rho^{-1})(x_1+x_2),\\
\tilde q_2(x,\psi)
&=q_2(x,\psi)+\tau(x).
\end{aligned}
\]
The density is invariant under column sign changes, so the fibre is represented
by \(0\leq\psi<\pi\) with \((dR)=d\psi\). Hence
\[
g_{n,\rho}(P_x)
=
\int_0^\pi f_{n,\rho}^{(2)}(U_xR_\psi)\,d\psi,
\quad
\overline g_{n,\rho}(x_1,x_2)=2\pi^2g_{n,\rho}(P_x),
\]
where the factor \(2\pi^2=\operatorname{vol}\mathbb G(4,2)\) converts
the density relative to \((dP)\) into the density relative to the
invariant probability measure.
Here \(f_{n,\rho}^{(2)}\) is the ordered frame density relative
to \((dH_2)\). The specialization of
Theorem~\ref{cor:joint-density-global} gives
\[
\begin{aligned}
f_{n,\rho}^{(2)}(U_xR_\psi)
&=
\mathcal N_{n,\rho}
\sum_{r=0}^{\infty}
\sum_{\substack{\mu\vdash r\\\ell(\mu)\leq2}}
\frac{C_\mu(\tfrac12\mathcal A_\perp(x))}
{r!C_\mu(I_2)}\\
&\qquad\times
\sum_{s=0}^{2}
\frac{s+1}{2^s}
\widetilde B_{\mu,(1^s)}^{(n)}
\widetilde{\mathcal M}_{r,s}^{(4)}
\bigl(q_1(x,\psi),\tilde q_2(x,\psi)\bigr),
\end{aligned}
\]
where
\(\mathcal N_{n,\rho}=B_2((n-2)/2,5/2)/
\{2^{2n}\rho^n\varGamma_4(n/2)\}\), and
\(\varGamma_4(n/2)=\pi^3\prod_{j=0}^{3}\varGamma((n-j)/2)\).
Moreover,
\[
\widetilde B_{\mu,(1^s)}^{(n)}
=
\sum_{\substack{\phi\vdash r+s\\\ell(\phi)\leq2}}
g_{\mu,(1^s)}^\phi
\frac{(5/2)_\phi}{((n+3)/2)_\phi}
C_\phi(I_2).
\]
Writing \(A_r=2n+r\),
\(\tilde\beta_{r,s}=(3n-5)/2+r+s\),
\(\tilde c_r=(3n+5)/2+r\), and
\(\widetilde Q=q_1+\tilde q_2\), the remaining ordered integral is
\[
\widetilde{\mathcal M}_{r,s}^{(4)}(q_1,\tilde q_2)
=
2^{A_r}\varGamma(A_r)
B(\tilde\beta_{r,s}+1,4-s)
\widetilde Q^{-A_r}
{}_2F_1\left(
A_r,4-s;\tilde c_r;
\frac{\tilde q_2}{\widetilde Q}
\right).
\]

After all remaining orbit angles are integrated out, the invariant
probability measure on \(\mathbb G(4,2)\) has density
\[
j_{4,2}(x_1,x_2)
=
\frac12
\frac{x_1-x_2}
{\sqrt{x_1x_2(1-x_1)(1-x_2)}},
\quad 1>x_1>x_2>0.
\]
Since \(T_{n,2}(P_0)=2n(2-x_1-x_2)\), its exact finite-sample null
distribution and quantile function are
\begin{align}
F_{n,\rho}(t)
&=
\frac12
\int_{\substack{1>x_1>x_2>0\\
	2n(2-x_1-x_2)\leq t}}
	\overline g_{n,\rho}(x_1,x_2)
	\frac{x_1-x_2}
	{\sqrt{x_1x_2(1-x_1)(1-x_2)}}
	\,dx_1\,dx_2,
	\label{eq:exact-grassmann-cdf}\\
	q_{n,\rho}(\beta)
	&=
	\inf\{t:F_{n,\rho}(t)\geq\beta\},
	\quad 0<\beta<1.
	\notag
\end{align}
Theorem~\ref{cor:joint-density-global} establishes the exact null law
in \eqref{eq:exact-grassmann-cdf}. We evaluate an equivalent integral
by importance quadrature using a proposal bank with seven components,
\(2^{16}\) Sobol points per component, and five randomizations. The grid
contains the \(0.95\)-quantile for every pair with
\(n\in\{7,10,20,40\}\) and
\(\rho\in\{1.25,1.5,2,2.5,3,4,5,6,8\}\), together with quantiles at
\(\beta\in\{0.01,0.02,\ldots,0.99\}\) for \(\rho=4\).

For \(0<\alpha<1\), define the upper tail critical value by
\(c_{n,\rho}(\alpha)=q_{n,\rho}(1-\alpha)\).

\begin{corollary}[Exact test and confidence region]
\label{cor:exact-subspace-test}
Suppose \(\rho>1\) is given. In model
\eqref{eq:two-level-covariance}, the rule
\[
\text{reject }H_0
\quad\Longleftrightarrow\quad
T_{n,2}(P_0)>c_{n,\rho}(\alpha)
\]
has level \(\alpha\) for every \(\lambda_2>0\). For the realized value
\(t_0=T_{n,2}(P_0)\), the exact \(p\)-value is
\(p(P_0)=1-F_{n,\rho}(t_0)\).
Inverting the tests gives the \((1-\alpha)\)-confidence region
\[
\mathcal C_{1-\alpha}(\rho)
=
\left\{
P\in\mathbb G(4,2):
n\|\widehat P_2-P\|_F^2\leq c_{n,\rho}(\alpha)
\right\}.
\]
It is a chordal ball centered at \(\widehat P_2\) and has exact
coverage \(1-\alpha\).
\end{corollary}

\begin{proof}
The conclusion follows because the null law is independent of the
orientation of \(P_0\) and the scale \(\lambda_2\). The continuity of
the Grassmann density eliminates boundary randomization. The
test inversion identity then gives
\(\Pr\{P_*\in\mathcal C_{1-\alpha}(\rho)\}=1-\alpha\).
\end{proof}

The prespecified \(\rho\) is essential. If external information only
specifies \(\rho\in\mathcal I\subset(1,\infty)\), the conservative
choices
\[
c_{n,\mathcal I}(\alpha)
=
\sup_{\rho\in\mathcal I}c_{n,\rho}(\alpha),
\quad
p(P_0,\mathcal I)
=
\sup_{\rho\in\mathcal I}
\{1-F_{n,\rho}(t_0)\}
\]
control size uniformly over \(\mathcal I\). An unrestricted nuisance
set reaching \(\rho=1\) is not meaningful for inference on \(P_*\),
because the target ceases to be identifiable at that boundary.

\subsection{Comparison with the weighted bootstrap}

We compare the exact law with the Gaussian weighted bootstrap of
\cite{NaumovSpokoinyUlyanov2019}. Conditional on the observed sample,
let
\[
\widehat{\varSigma}^{\circ}
=
\frac1n\sum_{i=1}^n w_iX_iX_i',
\quad
w_i\sim N(1,1),\quad 1\leq i\leq n,
\]
where the weights are independent. Let \(P_2^\circ\) be the projector
associated with the two largest
algebraic eigenvalues of \(\widehat{\varSigma}^{\circ}\). The
bootstrap statistic is
\[
T_{n,2}^{\circ}
=
n\|P_2^\circ-\widehat P_2\|_F^2.
\]
For the \(j\)th outer sample, write \(q_{\beta,j}^{\circ}\) for the
conditional \(\beta\)-quantile of \(T_{n,2}^{\circ}\). Its calibration
probability under the exact law and the corresponding error are
\[
C_{\beta,j}(n,\rho)
=
F_{n,\rho}(q_{\beta,j}^{\circ}),
\quad
E_{\beta,j}(n,\rho)
=
C_{\beta,j}(n,\rho)-\beta.
\]
Positive error means that, relative to an independent null draw, the
conditional bootstrap critical value is conservative.

Figure~\ref{fig:rank-two-qq} uses \(\rho=4\),
\(n\in\{7,10,20,40\}\), \(30\) outer samples, and \(1000\) Gaussian
weight vectors per sample. For the \(99\) levels
\(\beta\in\{0.01,0.02,\ldots,0.99\}\), the horizontal axis shows
\(q_{n,4}(\beta)\), evaluated from
\eqref{eq:exact-grassmann-cdf}.

\begin{figure}[htbp]
\centering
\includegraphics[width=1\linewidth]{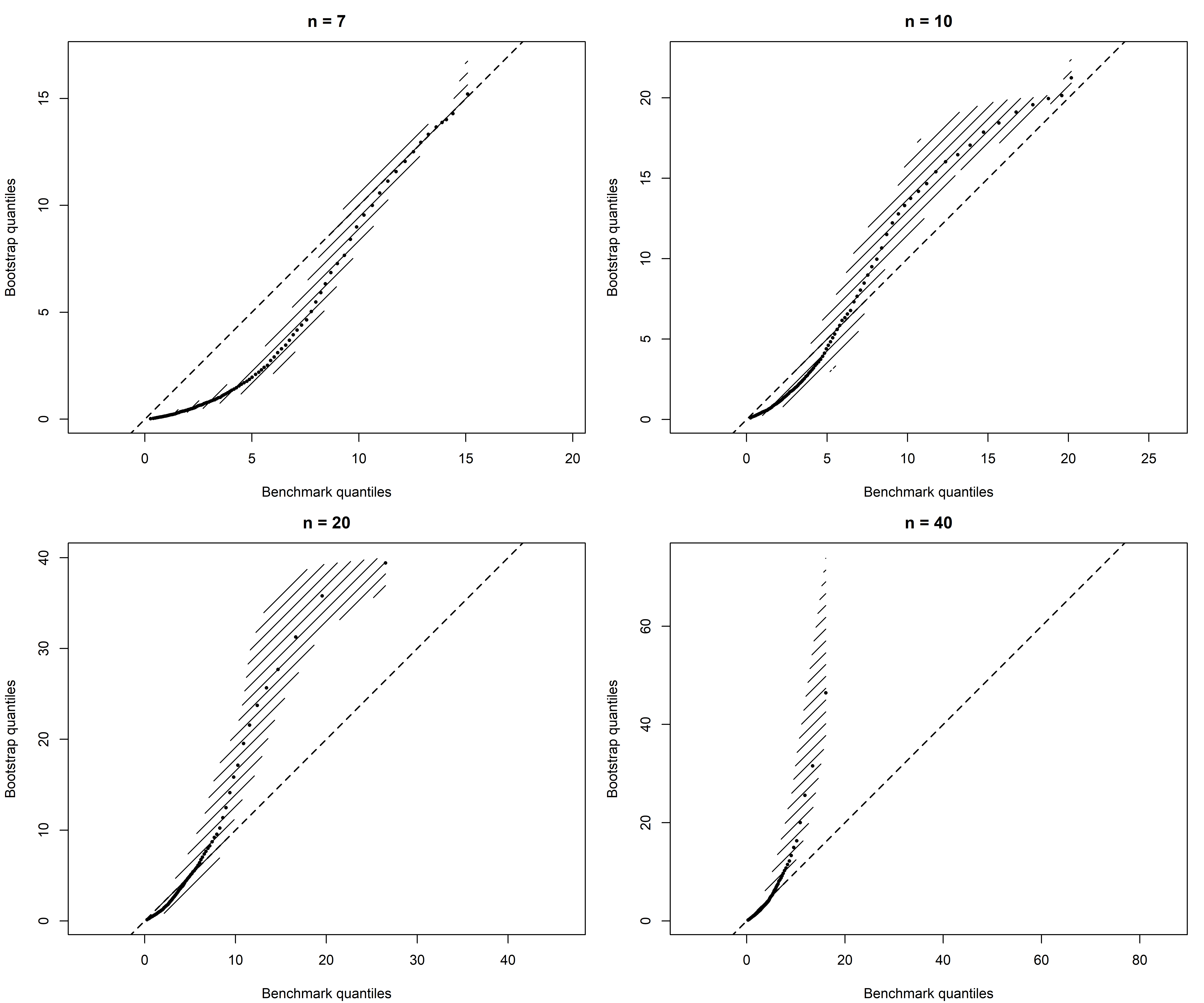}
\caption{Conditional Gaussian weighted bootstrap quantiles versus the
exact finite-sample null quantiles of \(T_{n,2}(P_0)\) at \(\rho=4\).
Each panel reports the indicated sample size using \(30\) outer samples
and \(1000\) weight vectors per sample. The solid curve is the median
conditional quantile, the hatched band is the interquartile range, and
the dashed diagonal marks exact agreement.}
\label{fig:rank-two-qq}
\end{figure}

Table~\ref{tab:rank-two-calibration} summarizes the comparison with
\cite{NaumovSpokoinyUlyanov2019}. At the nominal levels \(0.90\),
\(0.95\), and \(0.99\), it reports the median and interquartile range
of \(C_{\beta,j}(n,4)\) over the \(30\) outer samples. These summaries
remain subject to Monte Carlo error, especially at level \(0.99\).

\begin{table}[htbp]
\centering
\caption{Monte Carlo estimates of the calibration probabilities for
independent null draws under the Gaussian weighted bootstrap with
\(\rho=4\). Each entry is the median over \(30\) outer samples and uses \(1000\) Gaussian weight
vectors per sample. Interquartile ranges are in parentheses.}
\label{tab:rank-two-calibration}
\begin{tabular}{cccc}
\toprule
& \multicolumn{3}{c}{Nominal quantile level}\\
\(n\)
& \(0.90\)
& \(0.95\)
& \(0.99\)\\
\midrule
\(7\)
& \(0.896\;(0.875,\,0.925)\)
& \(0.952\;(0.940,\,0.963)\)
& \(0.991\;(0.984,\,0.998)\)
\\
\(10\)
& \(0.944\;(0.927,\,0.970)\)
& \(0.974\;(0.960,\,0.986)\)
& \(0.997\;(0.995,\,0.999)\)
\\
\(20\)
& \(0.966\;(0.922,\,0.981)\)
& \(0.989\;(0.980,\,0.996)\)
& \(0.999\;(0.997,\,0.999)\)
\\
\(40\)
& \(0.958\;(0.908,\,0.996)\)
& \(0.991\;(0.977,\,0.999)\)
& \(1.000\;(1.000,\,1.000)\)
\\
\bottomrule
\end{tabular}
\end{table}

The median calibration probabilities at level \(0.95\) are \(0.952\),
\(0.974\), \(0.989\), and \(0.991\) as \(n\) increases. All four
medians exceed \(0.95\), indicating conservative upper critical values
in this experiment. The comparison is descriptive rather than
mechanistic. Gaussian weights can make the weighted matrix indefinite,
and the conditional law is centered at \(\widehat{\varSigma}\), whose
eigengap is random. These features help explain the observed discrepancy.

\subsection{Power analysis results}

Let \(\delta_1\leq\delta_2\) be the principal angles between the true
subspace \(P_*\) and the null subspace \(P_0\). Orthogonal invariance
implies that power depends on the pair \((P_*,P_0)\) only through
these angles. The exact finite-sample power function is
\[
\begin{aligned}
\operatorname{Power}_{\alpha}(n,\rho,\delta_1,\delta_2)
&=
\Pr_{\varSigma(1,\rho,P_*)}
\{T_{n,2}(P_0)>c_{n,\rho}(\alpha)\}\\
&=
\int_{\mathbb G(4,2)}
\mathbf 1\{2n\mathfrak d^2(P,P_0)>c_{n,\rho}(\alpha)\}
g_{4,2}(P,\varSigma(1,\rho,P_*))(dP).
\end{aligned}
\]
For the isoclinic alternatives used in
Figures~\ref{fig:power-sample-size}--\ref{fig:power-principal-angle},
we take \(\delta_1=\delta_2=\delta\), represented by
\[
U_\delta=
\begin{pmatrix}\cos\delta\,I_2\\ \sin\delta\,I_2\end{pmatrix},
\quad
P_\delta=U_\delta U_\delta',
\quad
\mathfrak d(P_\delta,P_0)=\sqrt2\sin\delta.
\]
At every design point, we evaluate the critical value
\(c_{n,\rho}(0.05)=q_{n,\rho}(0.95)\) by quadrature. Each rejection probability is
estimated from \(5000\) independent alternative draws. Error bars are
pointwise \(95\%\) Wilson intervals. As
\(\delta\) tends to zero or \(\rho\) tends to one, power approaches the
nominal level. Across the plotted designs, increasing \(n\), \(\rho\),
or the principal angle separation is associated with higher power.

\begin{figure}[htbp]
\centering
\includegraphics[width=.82\linewidth]{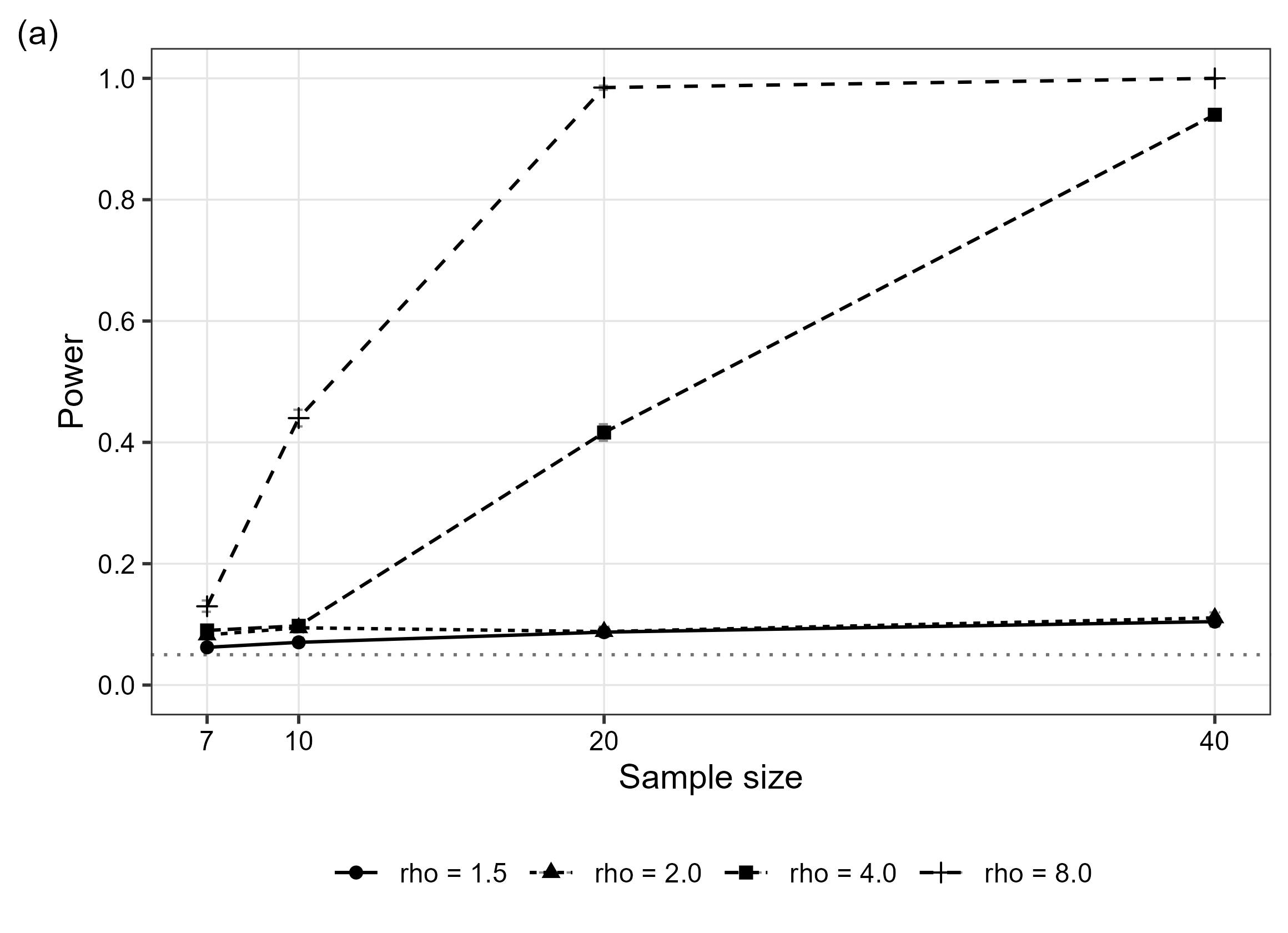}
\caption{Power versus sample size for the level \(0.05\) test of
\(H_0:P_*=P_0\). The common principal angle is \(20^\circ\),
\(n\in\{7,10,20,40\}\), and
\(\rho\in\{1.5,2,4,8\}\). Curves distinguish the four eigenvalue
ratios. Vertical bars are pointwise \(95\%\) Wilson intervals. The
dotted horizontal line marks the nominal level.}
\label{fig:power-sample-size}
\end{figure}

\begin{figure}[htbp]
\centering
\includegraphics[width=.82\linewidth]{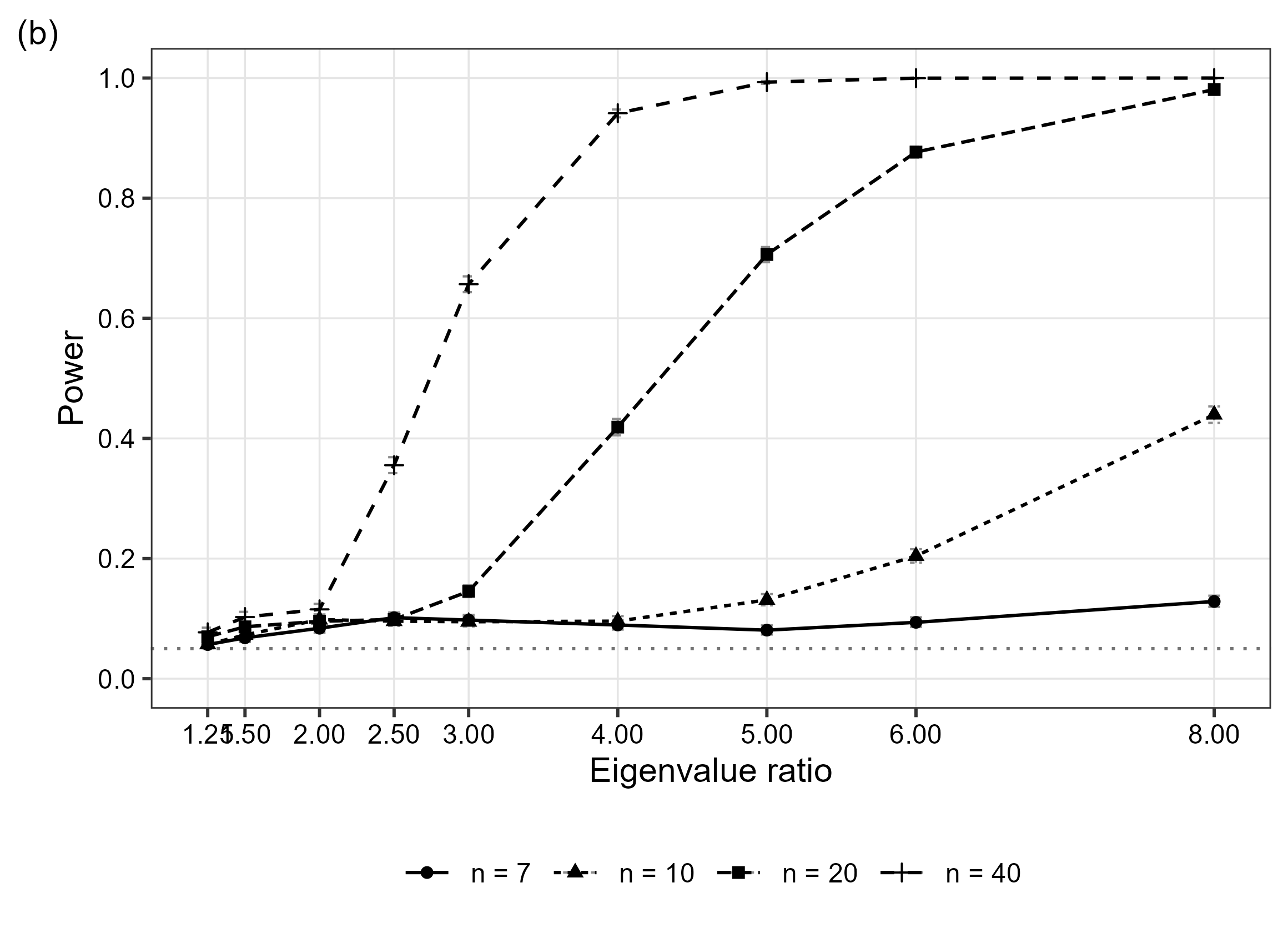}
\caption{Power versus the eigenvalue ratio for the level \(0.05\) test.
The common principal angle is \(20^\circ\),
\(n\in\{7,10,20,40\}\), and
\(\rho\in\{1.25,1.5,2,2.5,3,4,5,6,8\}\). Curves distinguish the four
sample sizes. Vertical bars are pointwise \(95\%\) Wilson intervals.
The dotted horizontal line marks the nominal level.}
\label{fig:power-eigenvalue-ratio}
\end{figure}

\begin{figure}[htbp]
\centering
\includegraphics[width=.82\linewidth]{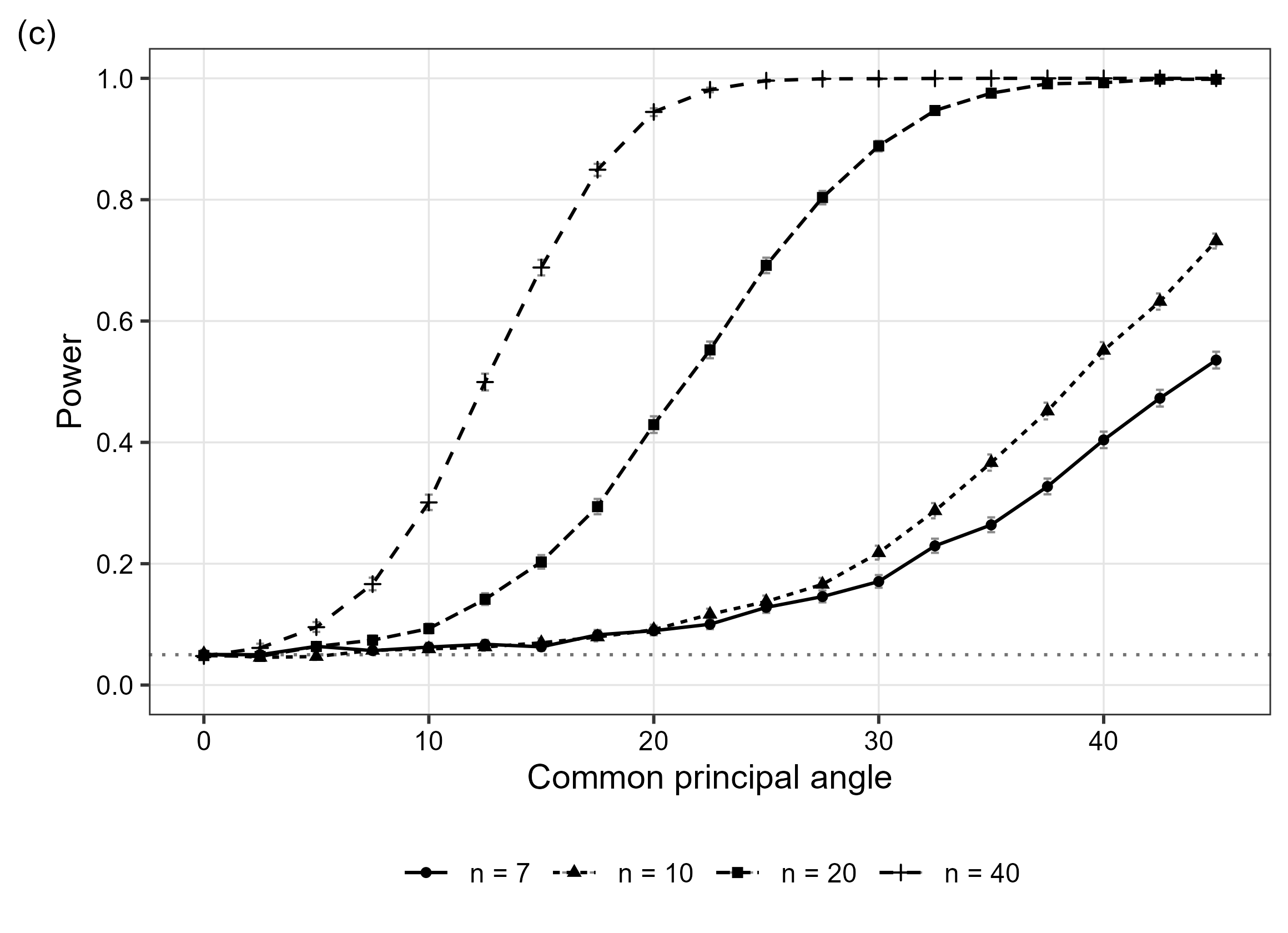}
\caption{Power versus the common principal angle for the level \(0.05\)
test. Here \(\rho=4\), \(n\in\{7,10,20,40\}\), and
\(\delta\in\{0,2.5,\ldots,45\}^\circ\). Curves distinguish the four
sample sizes. Vertical bars are pointwise \(95\%\) Wilson intervals.
The dotted horizontal line marks the nominal level.}
\label{fig:power-principal-angle}
\end{figure}

Power need not be determined by chordal distance alone. To isolate the
effect of principal angle geometry, Figure~\ref{fig:power-geometry}
compares \((\delta_1,\delta_2)=(0,\delta)\) with \((\eta,\eta)\), where
\(\sin\eta=\sin\delta/\sqrt2\). The two alternatives have the same
squared chordal distance,
\(\mathfrak d^2=\sin^2\delta=2\sin^2\eta\).

\begin{figure}[htbp]
\centering
\includegraphics[width=.82\linewidth]{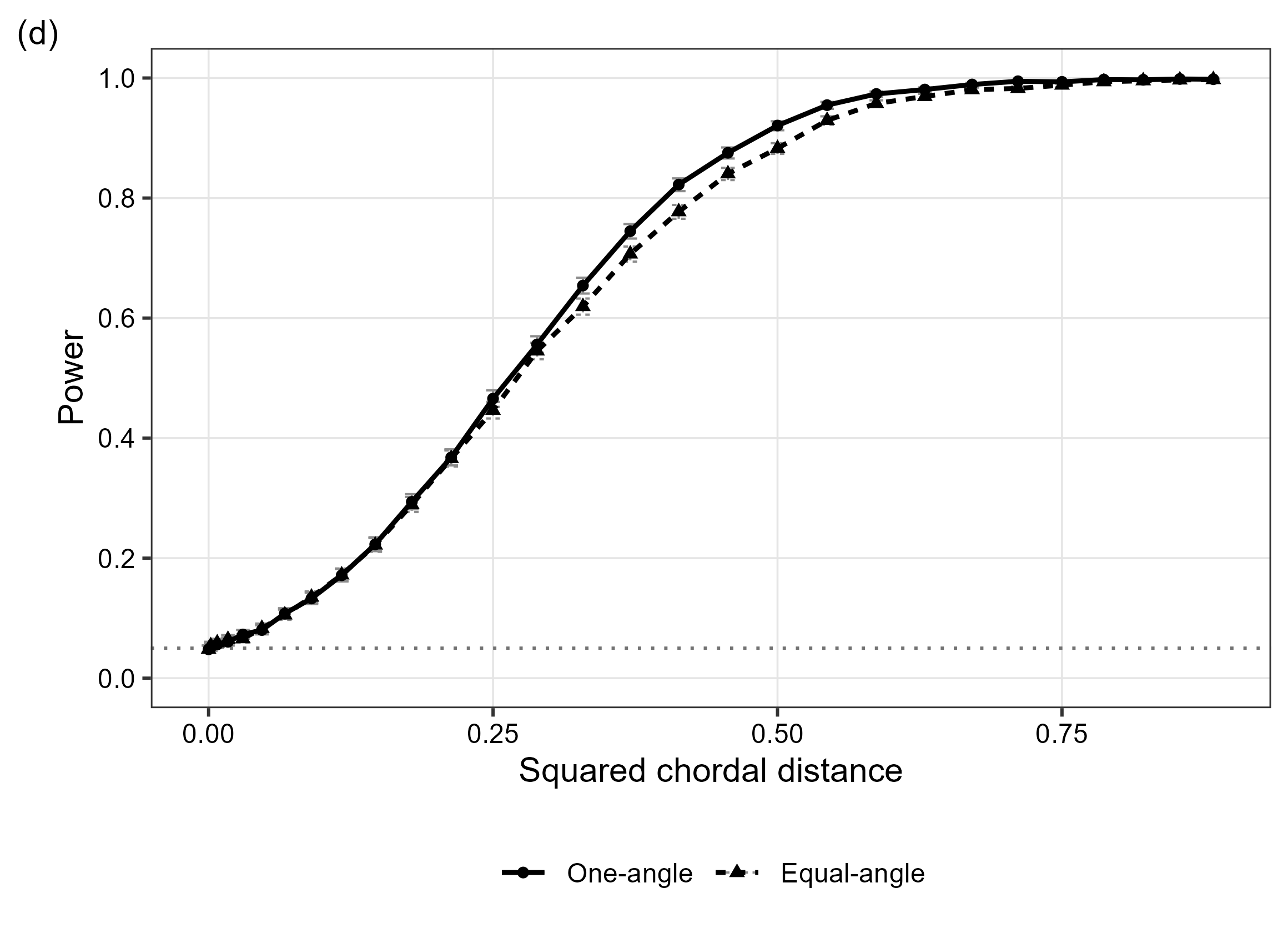}
\caption{Power for two principal angle configurations matched by squared
chordal distance. The settings are \(n=20\), \(\rho=4\),
and \(\delta\in\{0,2.5,\ldots,70\}^\circ\). The solid curve uses
\((0,\delta)\). The dashed curve uses \((\eta,\eta)\), where
\(\sin\eta=\sin\delta/\sqrt2\). Vertical bars are pointwise \(95\%\)
Wilson intervals. The dotted horizontal line marks the nominal level.}
\label{fig:power-geometry}
\end{figure}

The weighted bootstrap remains useful when the covariance is not
restricted to \eqref{eq:two-level-covariance}. In the present Gaussian
two-level model, however, Corollary~\ref{cor:exact-subspace-test}
provides a direct finite-sample calibration against which that
data-driven approximation can be assessed.

\appendix

\section{Derivation and Formulas for \texorpdfstring{\(k=2\)}{k=2}}
\label{app:k2}

This appendix derives the \(k=2\) specialization of the globally
convergent density in Theorem~\ref{cor:joint-density-global}. Put
\(m=p-2\), \(v=(n-p-1)/2\), \(u=(p-2)(n-2)/2\),
\(\tilde q_2=q_2+\operatorname{tr}\varSigma_m\), and
\(\widetilde Q=q_1+\tilde q_2\).

\subsection{Reduction of the partition sums}

The restriction \(\nu_1\leq1\) forces
\(\nu=(1^s)\), \(\nu'=(s)\), and \(0\leq s\leq m\). The complement
transformation in the proof of
Theorem~\ref{cor:joint-density-global} produces the positive
one-variable argument \(\lambda_2/(\lambda_1-\lambda_2)\). Hence
\[
J_{(s)}\left(
\frac{\lambda_2}{\lambda_1-\lambda_2};\frac12
\right)
=
\frac{(s+1)!}{2^s}
\left(
\frac{\lambda_2}{\lambda_1-\lambda_2}
\right)^s.
\]
After multiplication by \(1/s!\), the resulting positive coefficient
is \((s+1)/2^s\). Formula
\eqref{eq:calM-global} becomes
\begin{equation}
\begin{aligned}
\widetilde{\mathcal M}_{r,s}^{(p)}
&=
\mathscr D_2^{v+u+m+r+s}
\mathscr D_1^v
(\mathscr D_1-\mathscr D_2)^{m+1-s}
\varPhi_2(q_1,\tilde q_2)
\\
&=
\int_{\lambda_1>\lambda_2>0}
\lambda_1^v
\lambda_2^{v+u+m+r+s}
(\lambda_1-\lambda_2)^{m+1-s}
e^{-(q_1\lambda_1+\tilde q_2\lambda_2)/2}
d\lambda_1\,d\lambda_2.
\end{aligned}
\label{eq:k2-global-operator-app}
\end{equation}
Thus the finite sum associated with a fixed partition \(\mu\) is
\[
\sum_{s=0}^{m}
\frac{s+1}{2^s}
\widetilde B_{\mu,(1^s)}
\widetilde{\mathcal M}_{r,s}^{(p)}.
\]

\subsection{Gauss hypergeometric evaluation}

Set \(A_r=np/2+r\),
\(\tilde\beta_{r,s}=v+u+m+r+s\), and
\(\tilde c_r=v+u+2m+r+3\).
Using \(\lambda_1=x\) and \(\lambda_2=tx\) in
\eqref{eq:k2-global-operator-app} gives
\[
\widetilde{\mathcal M}_{r,s}^{(p)}
=
2^{A_r}\varGamma(A_r)
\int_0^1
\frac{t^{\tilde\beta_{r,s}}
(1-t)^{m+1-s}}
{(q_1+\tilde q_2t)^{A_r}}\,dt.
\]
Euler's integral and the Kummer--Pfaff transformation yield
\begin{equation}
\begin{aligned}
\widetilde{\mathcal M}_{r,s}^{(p)}
&=
2^{A_r}\varGamma(A_r)
B(\tilde\beta_{r,s}+1,m+2-s)
\widetilde Q^{-A_r}
\\
&\qquad\times
{}_2F_1\left(
A_r,m+2-s;
\tilde c_r;
\frac{\tilde q_2}{\widetilde Q}
\right).
\end{aligned}
\label{eq:k2-global-hypergeometric-app}
\end{equation}
The argument \(\tilde q_2/\widetilde Q\) lies strictly between zero
and one. This is the form used in the globally convergent series.

\subsection{The cases \texorpdfstring{\(p=3,4,5\)}{p=3,4,5}}

Write \(b_r=(pn-n-p-1)/2+r\) and
\(\tilde\beta_{r,s}=b_r+s\).
For \(p=3\), \(m=1\), \(b_r=n-2+r\), and
\(\tilde c_r=n+2+r\). The finite sum is
\[
\widetilde B_{\mu,\varnothing}
\widetilde{\mathcal M}_{r,0}^{(3)}
+
\widetilde B_{\mu,(1)}
\widetilde{\mathcal M}_{r,1}^{(3)}.
\]
For \(p=4\), \(m=2\), \(b_r=(3n-5)/2+r\), and
\(\tilde c_r=b_r+5\). The finite sum is
\[
\widetilde B_{\mu,\varnothing}
\widetilde{\mathcal M}_{r,0}^{(4)}
+
\widetilde B_{\mu,(1)}
\widetilde{\mathcal M}_{r,1}^{(4)}
+
\frac34\widetilde B_{\mu,(1^2)}
\widetilde{\mathcal M}_{r,2}^{(4)}.
\]
For \(p=5\), \(m=3\), \(b_r=2n-3+r\), and
\(\tilde c_r=b_r+6\). The finite sum is
\[
\begin{aligned}
&\widetilde B_{\mu,\varnothing}
\widetilde{\mathcal M}_{r,0}^{(5)}
+
\widetilde B_{\mu,(1)}
\widetilde{\mathcal M}_{r,1}^{(5)}
+
\frac34\widetilde B_{\mu,(1^2)}
\widetilde{\mathcal M}_{r,2}^{(5)}
\\
&\qquad+
\frac12\widetilde B_{\mu,(1^3)}
\widetilde{\mathcal M}_{r,3}^{(5)}.
\end{aligned}
\]
In each case, \(\widetilde{\mathcal M}_{r,s}^{(p)}\) is given by
\eqref{eq:k2-global-hypergeometric-app}.

\subsection{Elementary formulas when \texorpdfstring{\(n=p+1\)}{n=p+1}}

When \(n=p+1\), \(v=0\), and the spacing variables
\(x=\lambda_2\) and \(y=\lambda_1-\lambda_2\) separate the integral.
Consequently,
\[
\widetilde{\mathcal M}_{r,s}^{(p)}
=
\frac{
2^{A_r}
\varGamma(u+m+r+s+1)
\varGamma(m+2-s)
}{
\widetilde Q^{\,u+m+r+s+1}
q_1^{m+2-s}
}.
\]
For \(p=3\), \(n=4\), and \(A_r=6+r\), this reads
\[
\widetilde{\mathcal M}_{r,s}^{(3)}
=
\frac{
2^{6+r}\varGamma(3+r+s)\varGamma(3-s)
}{
\widetilde Q^{\,3+r+s}q_1^{3-s}
},
\quad s=0,1.
\]
For \(p=4\), \(n=5\), and \(A_r=10+r\),
\[
\widetilde{\mathcal M}_{r,s}^{(4)}
=
\frac{
2^{10+r}\varGamma(6+r+s)\varGamma(4-s)
}{
\widetilde Q^{\,6+r+s}q_1^{4-s}
},
\quad s=0,1,2.
\]
For \(p=5\), \(n=6\), and \(A_r=15+r\),
\[
\widetilde{\mathcal M}_{r,s}^{(5)}
=
\frac{
2^{15+r}\varGamma(10+r+s)\varGamma(5-s)
}{
\widetilde Q^{\,10+r+s}q_1^{5-s}
},
\quad s=0,1,2,3.
\]
These expressions involve only gamma functions and rational powers, and
the outer series remains globally convergent.
\section{Generalized Kostka Coefficients up to Degree 6}
\label{app:kostka}

This appendix lists the coefficients in the expansion of the
\(J\)-normalized Jack polynomials at \(\alpha=1/2\) in the Schur
basis
\[
J_\lambda(x;1/2)
=
\sum_{\mu\leq\lambda}
K_{\lambda\mu}(1/2)s_\mu(x).
\]
The matrices are upper triangular when the partitions are arranged in
a linear extension of dominance order.

\begin{table}[htbp]
\centering
\caption{Expansion coefficients \(K_{\lambda\mu}(1/2)\) for
\(|\lambda|=1,2,3\)}
\label{tab:jack_small}
\small
\begin{tabular}{lclcclccc}
\toprule
\multicolumn{2}{c}{$|\lambda|=1$}
& \multicolumn{3}{c}{$|\lambda|=2$}
& \multicolumn{4}{c}{$|\lambda|=3$} \\
\cmidrule(lr){1-2}\cmidrule(lr){3-5}\cmidrule(lr){6-9}
$\lambda\backslash\mu$ & $[1]$ & $\lambda\backslash\mu$ & $[2]$
& $[1^{2}]$ & $\lambda\backslash\mu$ & $[3]$ & $[2,1]$ & $[1^{3}]$ \\
\midrule
$[1]$ & 1 & $[2]$ & 3/2 & 1/2 & $[3]$ & 3 & 3/2 & 0 \\
&  & $[1^{2}]$ & 0 & 2 & $[2,1]$ & 0 & 5/2 & 1 \\
&  &  &  &  & $[1^{3}]$ & 0 & 0 & 6 \\
\bottomrule
\end{tabular}
\end{table}

\begin{table}[htbp]
\centering
\caption{Expansion coefficients \(K_{\lambda\mu}(1/2)\) for \(|\lambda|=4\)}
\label{tab:jack_n4}
\small
\begin{tabular}{lccccc}
\toprule
$\lambda\backslash\mu$ & $[4]$ & $[3,1]$ & $[2^{2}]$ & $[2,1^{2}]$ & $[1^{4}]$ \\
\midrule
$[4]$ & 15/2 & 9/2 & 3/2 & 0 & 0 \\
$[3,1]$ & 0 & 9/2 & 3/2 & 5/2 & 0 \\
$[2^{2}]$ & 0 & 0 & 15/2 & 5/2 & 3/2 \\
$[2,1^{2}]$ & 0 & 0 & 0 & 7 & 3 \\
$[1^{4}]$ & 0 & 0 & 0 & 0 & 24 \\
\bottomrule
\end{tabular}
\end{table}

\begin{table}[htbp]
\centering
\caption{Expansion coefficients \(K_{\lambda\mu}(1/2)\) for \(|\lambda|=5\)}
\label{tab:jack_n5}
\scriptsize
\begin{tabular}{lccccccc}
\toprule
$\lambda\backslash\mu$ & $[5]$ & $[4,1]$ & $[3,2]$ & $[3,1^{2}]$ & $[2^{2},1]$ & $[2,1^{3}]$ & $[1^{5}]$ \\
\midrule
$[5]$ & 45/2 & 15 & 15/2 & 0 & 0 & 0 & 0 \\
$[4,1]$ & 0 & 21/2 & 21/4 & 27/4 & 9/4 & 0 & 0 \\
$[3,2]$ & 0 & 0 & 45/4 & 15/4 & 25/4 & 5/2 & 0 \\
$[3,1^{2}]$ & 0 & 0 & 0 & 12 & 4 & 7 & 0 \\
$[2^{2},1]$ & 0 & 0 & 0 & 0 & 35/2 & 7 & 9/2 \\
$[2,1^{3}]$ & 0 & 0 & 0 & 0 & 0 & 27 & 12 \\
$[1^{5}]$ & 0 & 0 & 0 & 0 & 0 & 0 & 120 \\
\bottomrule
\end{tabular}
\end{table}

\begin{table}[htbp]
\centering
\caption{Expansion coefficients \(K_{\lambda\mu}(1/2)\) for \(|\lambda|=6\)}
\label{tab:jack_n6}
\scriptsize
\resizebox{\textwidth}{!}{%
\begin{tabular}{lccccccccccc}
	\toprule
	$\lambda\backslash\mu$ & $[6]$ & $[5,1]$ & $[4,2]$ & $[4,1^{2}]$ & $[3^{2}]$ & $[3,2,1]$ & $[3,1^{3}]$ & $[2^{3}]$ & $[2^{2},1^{2}]$ & $[2,1^{4}]$ & $[1^{6}]$ \\
	\midrule
	$[6]$ & 315/4 & 225/4 & 135/4 & 0 & 45/4 & 0 & 0 & 0 & 0 & 0 & 0 \\
	$[5,1]$ & 0 & 30 & 18 & 21 & 6 & 21/2 & 0 & 0 & 0 & 0 & 0 \\
	$[4,2]$ & 0 & 0 & 189/8 & 63/8 & 63/8 & 18 & 45/8 & 45/8 & 15/8 & 0 & 0 \\
	$[4,1^{2}]$ & 0 & 0 & 0 & 27 & 0 & 27/2 & 18 & 0 & 6 & 0 & 0 \\
	$[3^{2}]$ & 0 & 0 & 0 & 0 & 45 & 45/2 & 0 & 0 & 15 & 0 & 0 \\
	$[3,2,1]$ & 0 & 0 & 0 & 0 & 0 & 25 & 10 & 10 & 15 & 7 & 0 \\
	$[3,1^{3}]$ & 0 & 0 & 0 & 0 & 0 & 0 & 45 & 0 & 15 & 27 & 0 \\
	$[2^{3}]$ & 0 & 0 & 0 & 0 & 0 & 0 & 0 & 315/4 & 105/4 & 63/4 & 45/4 \\
	$[2^{2},1^{2}]$ & 0 & 0 & 0 & 0 & 0 & 0 & 0 & 0 & 63 & 27 & 18 \\
	$[2,1^{4}]$ & 0 & 0 & 0 & 0 & 0 & 0 & 0 & 0 & 0 & 132 & 60 \\
	$[1^{6}]$ & 0 & 0 & 0 & 0 & 0 & 0 & 0 & 0 & 0 & 0 & 720 \\
	\bottomrule
\end{tabular}
}
\end{table}

\bibliographystyle{imsart-nameyear}
\bibliography{bibtex}

\end{document}